\documentclass[a4paper, 10pt]{article}
\usepackage{courier}
\usepackage{blindtext} 
\usepackage{amsmath}
\usepackage{amssymb}
\usepackage[colorlinks=true,breaklinks]{hyperref}
\usepackage[hyphenbreaks]{breakurl}
\usepackage{xcolor}
\definecolor{c1}{rgb}{0,0,1}
\definecolor{c2}{rgb}{0,0.3,0.9}
\definecolor{c3}{rgb}{0.3,0.9}
\hypersetup{linkcolor={c1},citecolor={c2},urlcolor={c3}}
\usepackage{enumerate}
\usepackage{todonotes}
\usepackage{makeidx}
\makeindex
\usepackage[top=3.5cm,bottom=3.5cm,left=2.5cm,right=2.5cm]{geometry}
\usepackage{fancyhdr}

\def\XXint#1#2#3{{\setbox0=\hbox{$#1{#2#3}{\int}$ }
\vcenter{\hbox{$#2#3$ }}\kern-.6\wd0}}

\usepackage{amsthm}
\theoremstyle{plain}
\newtheorem{theorem}{Theorem}[section]

\theoremstyle{definition}
\newtheorem{definition}[theorem]{Definition}

\theoremstyle{lemma}
\newtheorem{lemma}[theorem]{Lemma}

\theoremstyle{Remark}
\newtheorem{Remark}[theorem]{Remark}

\theoremstyle{proposition}

\theoremstyle{corollary}

\theoremstyle{example}

\theoremstyle{assumption}

\theoremstyle{Résumé}

\usepackage{mathrsfs}
\usepackage{systeme}
\usepackage{colonequals}

\begin{document}
\pagestyle{empty}
\title{Long time behavior of discrete velocity kinetic equations}
\author{Gayrat Toshpulatov\thanks{\textit{Corresponding author.}\\ \textit{Address:}  Institut f\"ur Analysis und Numerik, Fachbereich Mathematik
und Informatik der Universit\"at M\"unster, Orl\'eans-Ring 10, 48149 M\"unster, Germany.\\
\textit{Email:} {\tt  gayrat.toshpulatov@uni-muenster.de}}}

\maketitle

\pagestyle{plain}

\textbf{Abstract.} 
We study long time behavior of some nonlinear discrete velocity kinetic equations in the one and three dimensions with periodic boundary conditions. 
We prove the exponential time decay of solutions towards the global equilibrium
in the $L^2$ space.  Our result holds for a wide class of interaction rates including the Goldstein-Taylor and Carleman equations, and
the estimates on the rate of convergence are explicit and constructive. The technique is based on the construction of suitable Lyapunov
functionals by modifying Boltzmann's entropy.
\\

\textbf{Résumé.}  Nous étudions le comportement en limite temporelle de certaines équations cinétiques non linéaires à vitesses discrètes en dimensions un et trois, avec des conditions aux limites périodiques. Nous démontrons la convergence exponentielle des solutions vers l’équilibre global dans l’espace $L^2$. Notre résultat s’applique à une  grande classe de taux d’interaction, comprenant notamment les équations de Goldstein–Taylor et de Carleman. De plus, les estimations du taux de convergence sont explicites et constructives. Notre approche repose sur la construction de fonctionnelles de Lyapunov adaptées, obtenues par modification de l’entropie de Boltzmann.
\\

\tableofcontents

\section{Introduction}
We consider the initial value problem 
\begin{equation}\label{1Eq}
\begin{cases}
\partial_t u+c\partial_x u=k(u,v,x)(v-u),\,\,\,\,&x\in \mathbb{T}, \,\,t>0,\\
\partial_t v-c\partial_x v=k(u,v,x)(u-v),\,\,\,\,\,&x\in \mathbb{T}, \,\,t>0,\\
u_{|t=0}=u_0,\,\,v_{|t=0}=v_0, &x\in \mathbb{T}.
\end{cases}
\end{equation}
This equation describes the evolution of the velocity distribution of a gas composed of two species that move at a constant speed $c>0$ in the $x-$direction. The variables $t\geq 0$ and $x\in \mathbb{T}=\mathbb{R}/\mathbb{Z}$ stand for time and position, respectively. The first unknown $u=u(t,x)$ denotes  the density function of particles moving at speed $c$ in the positive $x-$direction, while the second unknown $v=v(t,x)$ is   the density function of particles moving at speed $c$ in the negative $x-$direction. The function  $k(u,v,x)$ is  non-negative and describes the interaction between gas particles. For example, the case $k(u,v,x)=1$ is known as the Goldstein-Taylor model \cite{Gol}, \cite{Tay}, which can be converted to the telegraph equation. The case  $k(u,v,x)=u+v$ was introduced by Carlemann in \cite{Car} to describe the interaction of two type of particles such that an interaction of two particles of the first type results into two particles of the second type and vice  versa. 
The macroscopic variables for \eqref{1Eq} are the mass density $\rho\colonequals u+v$ and the momentum  $j\colonequals u-v.$ With this notations, \eqref{1Eq} is equivalent to 
\begin{equation}\label{1Eq0}
\begin{cases}
\partial_t \rho+c\partial_x j=0,\\
\partial_t j+c\partial_x \rho=-2kj.
\end{cases}
\end{equation}

Equation \eqref{1Eq}  can be extended to the three dimensions. In this case the gas particles can move in directions parallel to one of the axes $x_1,x_2,x_3$ either in the positive direction or the negative one at a constant speed $c>0$. Let $u_i=u_i(t,x_1,x_2,x_3)$ (respectively $v_i=v_i(t,x_1,x_2,x_3)$) denote the density of particles moving at speed $c$ in the positive (respectively negative) $x_i-$direction for $i\in \{1,2,3\}$. Then we have 
\begin{equation}\label{2Eq}
\begin{cases}
\partial_t u_i+c\partial_{x_i} u_i=k(u_1,u_2,u_3,v_1,v_2,v_3,x)(\rho-6u_i)\,\,\,\,&x\in \mathbb{T}^3, \,\,t>0,\\
\partial_t v_i-c\partial_{x_i}v_i =k(u_1,u_2,u_3,v_1,v_2,v_3,x)(\rho-6v_i), \,\,\,\,&x\in \mathbb{T}^3, \,\,t>0,\\
{u_i}_{|t=0}=u_{i,0},\,\,{v_{i}}_{|t=0}=v_{i,0}, \,\,\,\,&x\in \mathbb{T}^3,
\end{cases}
\end{equation}
where $\rho\colonequals \sum_{i=1}^3 (u_i+v_i)$ denotes mass density function as in the one dimensional case. 
We denote the mass densities in the $x_i-$directions as $$\rho_i\colonequals u_i+v_i,\,\,\,i\in \{1,2,3\}.$$
The momentum is defined as  $$j\colonequals \begin{pmatrix}
u_1-v_1\\
u_2-v_2\\
u_3-v_3
\end{pmatrix} \in \mathbb{R}^3.$$ With this notations, the following   macroscopic system   can be derived from \eqref{2Eq}
\begin{equation}\label{2Eq0}
\begin{cases}
\partial_t \rho+c\mathrm{div}_x j=0,\\
\partial_t j+c\begin{pmatrix}
\partial_{x_1}\rho_1\\ \partial_{x_2}\rho_2\\
\partial_{x_3}\rho_3
\end{pmatrix}=-6kj.
\end{cases}
\end{equation}

\bigskip

Whenever $u,v$ is a classical solution to \eqref{1Eq}, by integrating the first equation in \eqref{1Eq0} we obtain \emph{conservation of mass} 
\begin{equation*}
\int_{\mathbb{T}}(u(t,x)+v(t,x))dx=\int_{\mathbb{T}}(u_0(x)+v_0(x))dx,\,\,\,\,\forall\,t\geq 0.
\end{equation*}
Another interesting property of \eqref{1Eq} is that a family of functionals decay under the time evolution of the classical solution $u,v$. Let $\varphi$ be a $C^1$ convex function defined on  $[0,\infty)$ (e.g., $\varphi(u)=u^p,\, p\geq 1$). Multiplying the first and second equations in \eqref{1Eq} by $\varphi'(u)$ and $\varphi'(v),$ respectively, and integrating over $\mathbb{T},$  we obtain  
\begin{equation}\label{d/dt}
\frac{d}{dt}\int_{\mathbb{T}}(\varphi(u)+\varphi(v))dx=-\int_{\mathbb{T}}k(u,v,x)(u-v)(\varphi'(u)-\varphi'(v))dx.
\end{equation}
Since $\varphi$ is convex, the right hand side of this equation is non-positive. Hence, the integral  
$ \int_{\mathbb{T}}\big[\varphi(u)+\varphi(v)\big]dx$ is a non-increasing function of  $t\geq 0.$ 
The right hand side of \eqref{d/dt} vanishes if $u(t,x)=v(t,x)$ for all $t>0$ and $x\in \mathbb{T}.$ We call such functions as a \emph{local equilibrium}. One can check that a local equilibrium $u,v$ solves \eqref{1Eq} if 
 $u$ and $v$ are equal to a positive constant $m_{\infty}$ for all $t\geq 0$ and $x\in \mathbb{T}.$ We call such functions  as a \emph{global equilibrium}.

The three dimensional equation \eqref{2Eq} has the same properties as the one dimensional one \eqref{1Eq}. The initial mass is conserved
\begin{equation*}
\int_{\mathbb{T}^3} \sum_{i=1}^3(u_{i}(t,x)+v_{i}(t,x)) dx= \int_{\mathbb{T}^3} \sum_{i=1}^3(u_{0,i}(x)+v_{0,i}(x)) dx,\,\,\,\,\,\forall\, t\geq 0.
\end{equation*}
 As in the one dimensional case, for any convex function $\varphi:[0,\infty)\to \mathbb{R},$ one can check that
 \begin{align}\label{d/dt3}
 \frac{d}{dt} \int_{\mathbb{T}^3}\sum_{i=1}^3 (\varphi (u_i)+\varphi(v_i))dx=& -\frac{1}{2}\int_{\mathbb{T}^3}k\sum_{\substack{i,j=1\\ i\neq j}}^3(\varphi'(u_i)-\varphi'(u_j))(u_i-u_j) dx\nonumber\\
 &-\frac{1}{2}\int_{\mathbb{T}^3}k\sum_{\substack{i,j=1\\i\neq j}}^3(\varphi'(v_i)-\varphi'(v_j))(v_i-v_j) dx\nonumber \\
 &-\int_{\mathbb{T}^3}k\sum_{i, j=1}^3(\varphi'(u_i)-\varphi'(v_j))(u_i-v_j) dx\leq 0.
 \end{align}
 Hence,  the integral $\int_{\mathbb{T}^3}\sum_{i=1}^3 (\varphi (u_i)+\varphi(v_i))dx$
 is a non-increasing function of  $t\geq 0.$
From \eqref{d/dt3}, we also conclude that $u_i,v_i$ is a local equilibrium (i.e., the right hand side of \eqref{d/dt3} vanishes) if $u_i(t,x)=v_i(t,x)={\rho(t,x)}/{6}$ for all $ i\in\{1,2,3\},$ $t\geq 0,$ and $x\in \mathbb{T}^3,$ where we remind $\rho=\sum_{i=1}^3(u_i+v_i).$  $u_i,v_i$ is a global equilibrium if there exists $m_{\infty}>0$ and $u_i(t,x)=v_i(t,x)=m_{\infty}$ for all $ i\in\{1,2,3\},$ $t\geq 0,$ and $x\in \mathbb{T}^3.$ 

Let $u(t,x), v(t,x)$ be the solution of the one dimensional equation \eqref{1Eq}.  On the basis of the decay of   the integral $\int_{\mathbb{T}}(\varphi (u)+\varphi(v))dx$, 
 one can expect that $\int_{\mathbb{T}}(\varphi (u)+\varphi(v))dx$ decays  towards its minimum as $t\to \infty.$ From \eqref{d/dt} we see that $u,v$ solves \eqref{1Eq} and minimizes $\int_{\mathbb{T}}(\varphi (u)+\varphi(v))dx$ at the same time if $u=v=m_{\infty}$ for some constant $m_{\infty}$ and for all $t\geq 0,\, x\in \mathbb{T}.$ Hence, one can conjecture that   $u(t,x)$ and $ v(t,x)$ converge to  a constant $m_{\infty}$ as $t\to \infty.$ Since $u(t,x)+v(t,x)$  
  has the same as  $u_0(x)+v_0(x)$, 
   we expect that  $$m_{\infty}=\frac{1}{2}\int_{\mathbb{T}}(u_0(x)+v_0(x))dx.$$
   The same conjecture can be formulated for the three dimensional equation \eqref{2Eq}, i.e., $u_i(t,x)$ and $v_{i}(t,x)$ converge to $$m_{\infty}=\frac{1}{6}\int_{\mathbb{T}^3}\sum_{i=1}^3(u_{i,0}(x)+v_{i,0}(x))dx$$ as $t\to \infty$ for all $i\in \{1,2,3\}.$  Thus, it is an interesting problem to prove (or disprove) these convergences and to estimate the convergence rate. Our goal is to prove these convergences for the  equations \eqref{1Eq} and \eqref{2Eq} with explicit and constructive convergence rates.

There are many works concerning the equations \eqref{1Eq} and \eqref{2Eq}. We refer \cite{Gat, Cab, Review, L.T} for a general introduction to the discrete velocity kinetic equations.   Existence and uniqueness of weak solutions to \eqref{1Eq} and \eqref{2Eq} was proven by Lions and Toscani in \cite{L.T};   see Theorem \ref{exist} and \ref{exist3} below. The diffusion approximation of \eqref{1Eq} has been studied  mostly for the case $k(u,v,x)=(u+v)^{\alpha}$ with $\alpha\leq  1.$  In this case, the asymptotics leads to the fast diffusion equation if $\alpha\in [0,1],$ the slow diffusion equation if $\alpha\in [-1,0),$ and the porous media  equation if $\alpha<-1;$ see  \cite{P.T, L.T, S.V, G.S}. 

Concerning the large time behavior, there are studies only for  the Goldstein-Taylor equation (i.e., $k(u,v,x)=1$) and the Carleman equation (i.e., $k(u,v,x)=u+v$). In \cite{Des.Sal}, algebraic decay of solutions to its equilibrium in $L^2(\mathbb{T})$ was proved for  the Goldstein-Taylor equation with a spatially dependent interaction rate $k=k(x)\in H^1(\mathbb{T})$ and for  initial data $u_0,v_0\in H^2(\mathbb{T}).$ The authors used  the entropy method \cite{D.V} which consists of obtaining differential inequalities for the entropy and its dissipation. Bernard and  Salvarani in \cite{Ber.Sal}  proved exponential convergence to equilibrium  for the Goldstein-Taylor equation with  initial data $u_0,v_0\in H^1(\mathbb{T})$ by using its connection to the telegraph equation. Their decay rate is optimal and the  interaction rate $k=k(x)\geq 0$ does not have to be bounded from below by a positive constant. Similar  results were obtained in \cite{Ant} for the one and three dimensional Goldstein-Taylor equations by constructing suitable Lypunov functionals.
The long time behavior of $C^1$ solutions to the one dimensional Carleman equation was studied by Illner and Reed  in the whole domain \cite{Car2} and in the torus \cite{Car1}. The authors proved that in the whole domain there exists $C>0$ such that     $\max_{x\in \mathbb{R}}\{u(t,x),v(t,x)\}\leq \frac{C}{t} $  for all $t>0.$ They also proved in the torus that there exists $C>0$ such that $\max_{x\in \mathbb{T}}|u(t,x)-m_{\infty}|\leq \frac{C}{t} $ and $\max_{x\in \mathbb{T}}|v(t,x)-m_{\infty}|\leq \frac{C}{t} $ for all $t>0.$

In this paper, we shall improve these previous results. We consider global weak solutions established by Lions and Toscani in \cite{L.T}.  We establish exponential decay of  weak solutions to their global equilibrium  in
the $L^2$ space for a wide class of interaction rates $k$ in the one and three dimensional cases. Moreover, our decay rates are explicit and
constructive. The idea consists in constructing  Lypunov functionals for the one and three dimensional equations, which are equivalent to the square of the $L^2$ norm and and satisfy a Gr\"onwall inequality. This functionals are  constructed by modifying  Boltzmann's entropy.

 In recent years, many new so called  hypocoercivity methods have been introduced to study the long-time behavior of spatially inhomogeneous kinetic equations.
The challenge of hypocoercivity is to understand the interplay between the collision operator that provides dissipativity in the velocity variable and the transport one which is conservative, in order to obtain global dissipativity for the whole problem, see \cite{Guo, D.V, NeuMou, V, Har, AT1}. Most of these methods are developed for   linear kinetic equations and one has to work  in close-to-equilibrium regime for nonlinear kinetic equations.  
The main motivation of this work is to develop new hypocoercivity methods for  the nonlinear discrete velocity kinetic equations.     The method to be derived here does not require any close-to-equilibrium assumption on the initial data in the spirit of the recent works \cite{FPS, P.Tosh}.

The organization of this paper is as follows. In Section 2 we mention  existence and uniqueness results in literature, then we  state our main results. Section 3 is devoted to the one dimensional equation, where we first construct a Lypunov functional and establish its  equivalence to the $L^2$ norm. The proofs of our main result for the one dimensional equation is given at the end of Section 3. In Section 4 the analysis which is done for the one dimensional equation  is extended for the three dimensional equation and the proofs are given. Finally,
 we conclude and discuss possible extensions in Section 5.






\section{Main results}
Motivated by the work of Lions and Toscani in \cite{L.T}, we first introduce some conditions on the interaction rate $k.$
\begin{definition}\label{def} Let $k(u,v,x)$ (resp. $k(u_1,u_2,u_3,v_1,v_2,v_3,x)$) be a non-negative, measurable function defined for $u,v\in [0,\infty)$ (resp. $u_i,v_i\in [0,\infty),$ $i\in \{1,2,3\}$) and $x\in \mathbb{T}$ (resp. $x\in \mathbb{T}^3$). 
\begin{itemize}
\item[(i)]
We say $k(u,v,x)$ (resp. $k(u_1,...,v_3,x)$) is an \emph{ interaction rate of type 1}  if for every $\gamma>0$ there exists $\kappa_1=\kappa_1(\gamma)>0$ such that $0\leq k(u,v,x)\leq \kappa_1$ (resp. $0\leq k(u_1,...,v_3,x) \leq \kappa_1$) for all $0\leq  u,v\leq \gamma$ (resp. $0\leq  u_i,v_i\leq \gamma, $ $i\in \{1,2,3\}$) and $x\in \mathbb{T}$ (resp. $x\in \mathbb{T}^3$). 
\item[(ii)] We say $k(u,v,x)$ (resp. $k(u_1,...,v_3,x)$) is an   \emph{interaction rate of type 2} if $k(0,0,x)=+\infty$ (resp. $k(0,0,0,0,0,0,x)=+\infty$) and for every $\gamma>0$ there exists $\kappa_2=\kappa_2(\gamma)>0$ such that
$0\leq k(u,v,x)\leq \kappa_2$ (resp. $0\leq k(u_1,...,v_3,x)\leq \kappa_2$) for all $u,v\geq \gamma $ (resp. $u_i,v_i\geq \gamma,$ $i\in\{1,2,3\}$)  and $x\in \mathbb{T}$ (resp. $x\in \mathbb{T}^3$).
\item[(iii)] We say $k(u,v,x)$ (resp. $k(u_1,...,v_3,x)$) is an \emph{ interaction rate of type 3} if for every $\gamma_1>0$ and $\gamma_2>0$ there exist $\kappa_3=\kappa_3(\gamma_1,\gamma_2)>0$ and $\kappa_4=\kappa_4(\gamma_1,\gamma_2)>0$ such that
$\kappa_3\leq k(u,v,x)\leq \kappa_4$ (resp. $\kappa_3\leq k(u_1,...,v_3,x)\leq \kappa_4$) for all $\gamma_1\leq u,v\leq \gamma_2 $ (resp. $\gamma_1\leq u_i,v_i\leq \gamma_2, $ $i\in \{1,2,3\}$) and $x\in \mathbb{T}$ (resp. $x\in \mathbb{T}^3$).
\end{itemize}
\end{definition}
 For example, $k(u,v,x)=(u+v)^{\alpha}$ (or $k(u_1,u_2,u_3,v_1,v_2,v_3,x)=
 (u_1+u_2+u_3+v_1+v_2+v_3)^{\alpha}$) is of type 1 if $\alpha\geq 0,$  of type 2 if $\alpha<0,$ and type 3 for all $\alpha\in \mathbb{R}.$

Under the above assumptions on  the interaction rate $k,$ Lions and Toscani \cite{L.T} proved existence and uniqueness of global weak solutions to \eqref{1Eq}.
\begin{theorem}[Proposition 2.1 and 2.3 in \cite{L.T}]\label{exist} Let $0\leq u_0, v_0\in L^{\infty}(\mathbb{T}).$
\begin{itemize}
\item[(i)] Let $k$ be an  interaction rate of type 1. Then, there exists a unique global weak solution $u,v\in L^{\infty}([0,T)\times \mathbb{T})\cap C([0, T];L^p(\mathbb{T}))$ for all $T>0,$ $p\in [1, \infty).$ In addition, the following bound holds 
\begin{equation}\label{max}
\max\big\{||u(t)||_{L^{\infty}(\mathbb{T})}, ||v(t)||_{L^{\infty}(\mathbb{T})}\big\}\leq \max\big\{||u_0||_{L^{\infty}(\mathbb{T})}, ||v_0||_{L^{\infty}(\mathbb{T})}\big\},\,\,\,\,\forall\,t\geq 0.
\end{equation}
\item[(ii)] Let $k$ be an  interaction rate of type 2. Assume that there is $\delta>0$ such that $u_0, v_0\geq \delta$ on $\mathbb{T}.$ 
 Then, there exists a unique global weak solution $u,v\in L^{\infty}([0,T)\times \mathbb{T})\cap C([0, T];L^p(\mathbb{T}))$ for all $T>0,$ $p\in [1, \infty).$ In addition, this solution satisfies the bound \eqref{max} and  
\begin{equation}\label{min} \mathrm{ess\,inf}_{x\in \mathbb{T}}\big\{u(t,x),v(t, x)\big\}\geq \mathrm{ess\,inf}_{x\in \mathbb{T}}\big\{u_0(x),v_0(x)\big\},\,\,\,\,\forall\, t\geq 0
.\end{equation}
\end{itemize} 
\end{theorem}
We study the long time behavior of the weak solution obtained in the above theorem. Our first result is the following.
\begin{theorem}\label{main}
Let $k(u,v,x)$ be an interaction rate of type 3. Let $u,v\in L^{\infty}([0,\infty)\times \mathbb{T})$ be the weak solution of \eqref{1Eq} with the initial data $\delta \leq u_0, v_0\in L^{\infty}(\mathbb{T})$ for some $\delta >0.$ Let $M\colonequals \max\{||u_0||_{L^{\infty}},\,||v_0||_{L^{\infty}}\}$ and  $m_{\infty}\colonequals\frac{1}{2}\int_{\mathbb{T}}(u_0+v_0)dx.$ Then there exist positive constants $\Lambda=\Lambda (\delta, M,m_{\infty},c)$ and  $\lambda=\lambda(\delta, M,m_{\infty},c)$ such that  
\begin{align*}
\int_{\mathbb{T}}\left({u(t,x)}-{m_{\infty}}\right)^2dx&+\int_{\mathbb{T}}\left({v(t,x)}-{m_{\infty}}\right)^2dx\\ \leq & \Lambda e^{-2\lambda t} \left[\int_{\mathbb{T}}\left({u_0(x)}-{m_{\infty}}\right)^2dx+\int_{\mathbb{T}}\left({v_0(x)}-{m_{\infty}}\right)^2dx
\right]
\end{align*}
holds for all $t\geq 0.$
\end{theorem}
Theorem \ref{main} holds for a large class of interaction rates $k$, for example,  $k(u,v,x)=k_1(x)(u+v)^{\alpha}$  for  $\alpha\in \mathbb{R}$ and  $k_1\in L^{\infty}(\mathbb{T})$ with $\mathrm{ess\,inf}_{x\in \mathbb{T}}k_1(x)>0.$ The lower bound on the initial data in  Theorem \ref{main} is required because the interaction rate $k(u,v,x)$ can have singularity when $u=v=0,$ for example,  as above $k(u,v,x)=k_1(x)(u+v)^{\alpha}$ with   $\alpha<0.$
 Our next result shows that this lower bound on the initial data can be removed under a lower growth assumption  on $k.$
\begin{theorem}\label{main1}
Let $k(u,v,x)$ be an interaction rate of type 1. Assume $k(u,v,x)\geq k_1(x)(u+v)^{\alpha}$ holds  for some $\alpha\in [0,1]$ and $k_1\in L^{\infty}(\mathbb{T})$ with $\mathrm{ess\,inf}_{x\in \mathbb{T} } k_1(x)>0.$   Let $u, v \in L^{\infty}([0,\infty)\times \mathbb{T})$ be the weak solution of \eqref{1Eq} with the initial data $0 \leq u_0, v_0\in L^{\infty}(\mathbb{T}).$  Let $M\colonequals \max\{||u_0||_{L^{\infty}},\,||v_0||_{L^{\infty}}\}$ and  $m_{\infty}\colonequals\frac{1}{2}\int_{\mathbb{T}}(u_0+v_0)dx>0.$  Then there exist positive constants $\Lambda_{\alpha}=\Lambda_{\alpha} ( M, m_{\infty}, c, \mathrm{ess\,inf}_{x\in \mathbb{T} } k_1(x))$ and  $\lambda_{\alpha}=\lambda_{\alpha}( M, m_{\infty}, c, \mathrm{ess\,inf}_{x\in \mathbb{T} } k_1(x))$ such that  
\begin{align*}
\int_{\mathbb{T}}\left({u(t,x)}-{m_{\infty}}\right)^2dx&+\int_{\mathbb{T}}\left({v(t,x)}-{m_{\infty}}\right)^2dx\\ \leq & \Lambda_{\alpha} e^{-2\lambda_{\alpha} t} \left[\int_{\mathbb{T}}\left({u_0(x)}-{m_{\infty}}\right)^2dx+\int_{\mathbb{T}}\left({v_0(x)}-{m_{\infty}}\right)^2dx
\right]
\end{align*}
holds for all $t\geq 0.$
\end{theorem}
In Theorem \ref{main1}  the interaction rate $k(u,v,x) $ can be \emph{degenerate} in the sense that it can vanish at the points where $u=v=0.$ 
For example, Theorem \ref{main1} holds for $k(u,v,x)=k_1(x)(u+v)^{\alpha}$ with $\alpha\in [0,1],$  $k_1\in L^{\infty}(\mathbb{T}),$  $\mathrm{ess\,inf}_{x\in \mathbb{T} } k_1(x)>0.$ This case includes the  the Goldstein-Taylor equation if $\alpha=0$ and the Carlemann equation if $\alpha=1$.
\bigskip
 
Next, we study the three dimensional equation \eqref{2Eq}. 
Existence and uniqueness of global weak solutions can obtained as in Theorem \ref{exist}; we refer \cite{L.T} for more details. 
\begin{theorem}\label{exist3} Let $0\leq u_{0,i}, v_{0,i}\in L^{\infty}(\mathbb{T}^3)$ for $i\in \{1,2,3\}.$
\begin{itemize}
\item[(i)] Let $k$ be an  interaction rate of type 1. Then, there exists a unique global weak solution $u_i,v_i\in L^{\infty}([0,T)\times \mathbb{T}^3)\cap C([0, T];L^p(\mathbb{T}^3))$ for all $T>0,$ $p\in [1, \infty).$  In addition, the following bound holds
\begin{equation}\label{max3}
\max_{i\in \{1,2,3\}}\big\{||u_i(t)||_{L^{\infty}(\mathbb{T}^3)}, ||v_i(t)||_{L^{\infty}(\mathbb{T}^3)}\big\}\leq \max_{i\in \{1,2,3\}}\big\{||u_{0,i}||_{L^{\infty}(\mathbb{T}^3)}, ||v_{0,i}||_{L^{\infty}(\mathbb{T}^3)}\big\},\,\,\,\,\forall\,t\geq 0.
\end{equation}
\item[(ii)] Let $k$ be an  interaction rate of type 2. Assume that there is $\delta>0$ such that $u_{0,i}, v_{0,i}\geq \delta$ on $\mathbb{T}^3$ for all $i\in \{1,2,3\}.$ 
 Then, there exists a unique global weak solution $u_i,v_i\in L^{\infty}([0,T)\times \mathbb{T}^3)\cap C([0, T];L^p(\mathbb{T}^3))$ for all $T>0,$ $p\in [1, \infty).$ In addition, this solution satisfies the bound \eqref{max3} and  
\begin{equation}\label{min3} \mathrm{ess\,inf}_{i\in \{1,2,3\}, x\in \mathbb{T}^3
}\big\{u_i(t,x),v_i(t, x)\big\}\geq \mathrm{ess\,inf}_{i\in \{1,2,3\}, x\in \mathbb{T}}\big\{u_{0,i}(x),v_{0,i}(x)\big\},\,\,\,\,\forall\, t\geq 0
.\end{equation}
\end{itemize} 
\end{theorem}

Our next result is the generalization of Theorem \ref{main} to the three dimensional equation.
\begin{theorem}\label{main3}
Let $k(u_1,u_2,u_3,v_1,v_2,v_3,x)$ be an interaction rate of type 3. Let $u_i,v_i\in L^{\infty}([0,\infty)\times \mathbb{T}^3),$ $i\in\{1,2,3\}$ be the weak solution of \eqref{2Eq} with the initial data $\delta \leq u_{0,i}, v_{0,i}\in L^{\infty}(\mathbb{T}^3)$ for some $\delta >0.$  Let $\displaystyle M\colonequals \max_{i\in\{1,2,3\}}\{||u_{0,i}||_{L^{\infty}},\,||v_{0,i}||_{L^{\infty}}\}$ and  
$m_{\infty}=\frac{1}{6}\int_{\mathbb{T}^3}\sum_{i=1}^3(u_{i,0}(x)+v_{i,0}(x))dx.$  Then there exist positive constants $\tilde{\Lambda}=\tilde{\Lambda} (\delta, M,m_{\infty},c)$ and  $\tilde{\lambda}=\tilde{\lambda}(\delta, M,m_{\infty},c)$ such that  
\begin{align*}
\sum_{i=1}^3\int_{\mathbb{T}^3}\left({u_i(t,x)}-{m_{\infty}}\right)^2dx&+ \sum_{i=1}^3\int_{\mathbb{T}^3}\left({v_i(t,x)}-{m_{\infty}}\right)^2dx \\ \leq & \tilde{\Lambda} e^{-2\tilde{\lambda} t} \left[ \sum_{i=1}^3\int_{\mathbb{T}^3}\left({u_{0,i}(x)}-{m_{\infty}}\right)^2dx+\sum_{i=1}^3\int_{\mathbb{T}^3}\left({v_{0,i}(x)}-{m_{\infty}}\right)^2dx
\right]
\end{align*}
holds for all $t\geq 0.$
\end{theorem}

 Similar to Theorem \ref{main1} we show that the lower bound on the initial data in Theorem \ref{main3} can be removed  under a lower growth assumption on  the interaction rate $k.$ 
\begin{theorem}\label{main13}
Let $k(u_1,u_2,u_3,v_1,v_2,v_3,x)$ be an interaction rate of type 1. Assume $$k(u_1,u_2,u_3,v_1,v_2,v_3,x)\geq k_1(x)
 (u_1+u_2+u_3+v_1+v_2+v_3)^{\alpha}$$ holds  for some $\alpha\in [0,1]$ and $k_1\in L^{\infty}(\mathbb{T}^3)$ with $\mathrm{ess\,inf}_{x\in \mathbb{T}^3 } k_1(x)>0.$   Let $u_i, v_i \in L^{\infty}([0,\infty)\times \mathbb{T}^3),$ $i\in\{1,2,3\}$ be the weak solution of \eqref{2Eq} with the initial data $0 \leq u_{0,i}, v_{0,i}\in L^{\infty}(\mathbb{T}^3).$  Let $\displaystyle M\colonequals \max_{i\in\{1,2,3\}}\{||u_{0,i}||_{L^{\infty}},\,||v_{0,i}||_{L^{\infty}}\}$ and $m_{\infty}=\frac{1}{6}\int_{\mathbb{T}^3}\sum_{i=1}^3(u_{i,0}(x)+v_{i,0}(x))dx>0.$  Then there exist positive constants $\tilde{\Lambda}_{\alpha}=\tilde{\Lambda}_{\alpha} ( M, m_{\infty}, c, \mathrm{ess\,inf}_{x\in \mathbb{T}^3 } k_1(x))$ and  $\tilde{\lambda}_{\alpha}=\tilde{\lambda}_{\alpha}( M, m_{\infty}, c, \mathrm{ess\,inf}_{x\in \mathbb{T}^3 } k_1(x))$ such that  
\begin{align*}
\sum_{i=1}^3\int_{\mathbb{T}^3}\left({u_i(t,x)}-{m_{\infty}}\right)^2dx&+ \sum_{i=1}^3\int_{\mathbb{T}^3}\left({v_i(t,x)}-{m_{\infty}}\right)^2dx \\ \leq & \tilde{\Lambda}_{\alpha} e^{-2\tilde{\lambda}_{\alpha} t} \left[ \sum_{i=1}^3\int_{\mathbb{T}^3}\left({u_{0,i}(x)}-{m_{\infty}}\right)^2dx+\sum_{i=1}^3\int_{\mathbb{T}^3}\left({v_{0,i}(x)}-{m_{\infty}}\right)^2dx
\right]
\end{align*}
holds for all $t\geq 0.$
\end{theorem}

\section{One dimensional model}
 
\subsection*{3.1 Lyapunov functionals}

As we mentioned in the introduction, if $u,v$ is a classical solution to \eqref{1Eq}, then,   for any convex function $\varphi,$ the integral  
$ \int_{\mathbb{T}}\big[\varphi(u(t,x))+\varphi(v(t,x))\big]dx$ is non-increasing in $t\geq 0.$     We consider a specific case, i.e., $\varphi (u)=u\log{\left(\frac{u}{m_{\infty}}\right)}-{u}+{m_{\infty}}$ with continuous extension to 0. This corresponds to  Boltzmann's entropy
\begin{equation}\label{H}\mathrm{H}[u,v]\colonequals \int_{\mathbb{T}}\left(\frac{u}{m_{\infty}}\log{\left(\frac{u}{m_{\infty}}\right)}-\frac{u}{m_{\infty}}+1\right)m_{\infty}dx+\int_{\mathbb{T}}\left(\frac{v}{m_{\infty}}\log{\left(\frac{v}{m_{\infty}}\right)}-\frac{v}{m_{\infty}}+1\right)m_{\infty}dx
\end{equation}
defined for non-negative functions $u,v \in L^{\infty}( \mathbb{T}) $ with  $m_{\infty}\colonequals \frac{1}{2}\int_{\mathbb{T}}(u+v)dx>0.$ 
This functional is non-negative, which follows from the fact that  $ s\log{s}-s+1$ is non-negative for $s\in [0,\infty).$  We actually show that $\mathrm{H}$ is equivalent to the square of the $L^2$ norm. 
\begin{lemma}\label{Heqv}
Let $0 \leq u, v\in L^{\infty}(\mathbb{T}),$  $M\colonequals \max\{||u||_{L^{\infty}},\,||v||_{L^{\infty}}\},$ and $m_{\infty}\colonequals \frac{1}{2}\int_{\mathbb{T}}(u+v)dx>0.$  Then the following inequalities hold
\begin{equation}\label{M/m}
 \frac{1}{2M}\int_{\mathbb{T}}\big(({u}-{m_{\infty}})^2+({v}-{m_{\infty}})^2\big) dx\leq \mathrm{H}[u,v]\leq  \frac{1}{m_{\infty}}\int_{\mathbb{T}}\big(({u}-{m_{\infty}})^2+({v}-{m_{\infty}})^2\big) dx.
\end{equation}
\end{lemma} 
\begin{proof}
Consider $f(u)\colonequals \frac{u}{m_{\infty}}\log{\left(\frac{u}{m_{\infty}}\right)}-\frac{u}{m_{\infty}}+1.$ We have  $$f'(u)=\frac{1}{m_{\infty}}\log{\left(\frac{u}{m_{\infty}}\right)}\,\,\,\,\, \text{and }\,\,\,  f''(u)=\frac{1}{m_{\infty}u}\,\,\, \text{for}\,\,\, u>0.$$ 
Since $f(m_{\infty})=f'(m_{\infty})=0,$ the mean value theorem implies that there exists $\theta\in [0,1]$ such that
$$f(u)=f(u)-f(m_{\infty})=\frac{(u-m_{\infty})^2}{2m_{\infty}(\theta u+(1-\theta)m_{\infty})}\geq  \frac{1}{2m_{\infty}M}\left({u}-{m_{\infty}}\right)^2.$$
This proves the inequality on the left hand side of \eqref{M/m}.

Next, we define $$g(u)\colonequals \frac{u}{m_{\infty}}\log{\left(\frac{u}{m_{\infty}}\right)}-\frac{u}{m_{\infty}}+1-\left(\frac{u}{m_{\infty}}-1\right)^2.$$ We have $$g'(u)=\frac{1}{m_{\infty}}\log{\left(\frac{u}{m_{\infty}}\right)}-\frac{2}{m_{\infty}} \left(\frac{u}{m_{\infty}}-1\right)\,\,\,\,\, \text{and }\,\,\,\,\, g''(u)=\frac{1}{m_{\infty}u}-\frac{2}{m^2_{\infty}}\,\,\, \text{for}\,\,\, u>0.$$
One can check that there exists $\bar{u}\in (0,m_{\infty})$ such that $g'(\bar{u})=0,$ $g'(u)<0  $ for $u\in (0,\bar{u}),$  $g'(u)>0$ for $u\in (\bar{u},m_{\infty}),$ and $g'(u)<0$ for $u>m_{\infty}.$ This fact and $g''(m_{\infty})=-\frac{1}{m^2_{\infty}}<0$ show that $g$ attains its  maximum  at $m_{\infty}.$ Therefore, we have
$$\frac{u}{m_{\infty}}\log{\left(\frac{u}{m_{\infty}}\right)}-\frac{u}{m_{\infty}}+1- \left(\frac{u}{m_{\infty}}-1\right)^2\leq g(m_{\infty})=0,\,\,\,\, \forall\, u\geq 0.$$
This proves the inequality on the right hand side of \eqref{M/m}.
\end{proof}
\begin{lemma}\label{a.e.H}
$(i)$ Let $k(u,v,x)$ be an interaction rate of type 2.  Let $u,v\in L^{\infty}([0,\infty)\times \mathbb{T})$ be the weak solution of \eqref{1Eq} with the initial data 
$\delta \leq u_0, v_0\in L^{\infty}(\mathbb{T})$ for some $\delta>0$. Let  $m_{\infty}\colonequals\frac{1}{2}\int_{\mathbb{T}}(u_0+v_0)dx.$ Then, for almost every $t\in (0,\infty),$ we have
\begin{equation*}
\frac{d}{dt}\mathrm{H}[u(t),v(t)]= -\int_{\mathbb{T}}k(u,v,x)\left(\log{u}-\log{v}\right)(u-v)dx\leq 0.
\end{equation*}
$(ii)$ Let $k(u,v,x)$ be an interaction rate of type 1. Assume $k(u,v,x)\geq k_1(x)(u+v)^{\alpha}$ holds  for some $\alpha\in [0,1]$ and $k_1\in L^{\infty}(\mathbb{T})$ with $\mathrm{ess\,inf}_{x\in \mathbb{T} } k_1(x)>0.$   Let $u,v\in L^{\infty}([0,\infty)\times \mathbb{T})$ be the weak solution of \eqref{1Eq} with the initial data 
$0 \leq u_0, v_0\in L^{\infty}(\mathbb{T}).$  Let  $M\colonequals \max\{||u_0||_{L^{\infty}},\,||v_0||_{L^{\infty}}\}$ and  $m_{\infty}\colonequals\frac{1}{2}\int_{\mathbb{T}}(u_0+v_0)dx>0.$ Then, for any $T>0,$ we have
\begin{equation*}
\mathrm{H}[u(T),v(T)]-\mathrm{H}[u_{0},v_{0}]\leq  - (2M)^{\alpha-1} \mathrm{ess\,inf}_{x\in \mathbb{T}}k_1(x)\int_{0}^{T}\int_{\mathbb{T}} (u-v)^2 dxdt.
\end{equation*}
\end{lemma}
\begin{proof} $(i)$
Theorem \ref{exist} provides that the solution  $u,v$ satisfies
\begin{equation}\label{bound}\delta\leq u(t,x),v(t,x)\leq \max\{||u_0||_{L^{\infty}}, ||v_0||_{L^{\infty}}\},\,\,\,\,\,\, \forall\,t\geq 0, \,\forall\, x\in \mathbb{T}.
\end{equation} 
We have 
\begin{equation}\label{dist1}
\int_0^{\infty}\int_{\mathbb{T}} u(\partial_t \psi+c\partial_x \psi)dxdt+\int_{\mathbb{T}}\psi_{|t=0}u_0dx+\int_0^{\infty}\int_{\mathbb{T}} k(u,v,x)(u-v)\psi dxdt=0
\end{equation}
and \begin{equation}\label{dist2}
\int_0^{\infty}\int_{\mathbb{T}} v(\partial_t \psi-c\partial_x \psi)dxdt+\int_{\mathbb{T}}\psi_{|t=0}v_0dx+\int_0^{\infty}\int_{\mathbb{T}} k(u,v,x)(v-u)\psi dxdt=0
\end{equation}
for all $\psi\in C^{\infty}_{c}([0,\infty)\times \mathbb{T}).$  We consider the function $r\in C^{\infty}_c(\mathbb{R})$ given by
$$r(x)\colonequals \begin{cases}
e^{\frac{1}{x^2-1}},\, \, \,  &\text{if}\, \, \, \, |x|< 1\\
0, &\text{if}\,\, \, \, |x|>1.
\end{cases}$$
We define a sequence of smooth functions $r_n\in C_c^{\infty}(\mathbb{R}), \,  n\in \mathbb{N},$ given by
$$r_n(x)\colonequals \frac{1}{\int_{\mathbb{R}}r(x')dx'}n r(nx).$$
Let $\tilde{r}_n\in C^{\infty}(\mathbb{T})$ be a periodic function such that $\tilde{r}_n(x)=r_n(x)$ for $x\in [-1/2,1/2).$
Let $\{\tau_n\}_{n=1}^{\infty}$  be a sequence of functions on $\mathbb{R}$  such that  
$$\tau_n\in  C^{\infty}_c(\mathbb{R}),\,\,\, \mathrm{supp}\,\tau_n\subset (-1/n, 0),\,\,\,\int_{\mathbb{R}}\tau_n(t)dt=1, \,\,\,\tau_n\geq 0 \,\,\text{on}\,\,\mathbb{R}.$$ 
Then  $\eta_n(t,x)\colonequals \tau_n(t)\tilde{r}_n(x) $ forms a sequence of smooth mollifiers (see \cite[Section 4.4]{Brezis})\footnote{Actually, $\tilde{r}_n$ is smooth only for $n\geq 2,$ but this won't be a problem in the sequel.}. 
 We define
\begin{equation*}
u_n\colonequals \eta_n*u\,\,\,\text{and}\,\,\,v_n\colonequals \eta_n*v
\end{equation*}
where  $u$ and $v$ are extended to $t\in (-\infty,0)$ by making them equal to zero.
Young's inequality (see \cite[Proposition 8.9 and 8.10]{Foll}) provides  $$u_n, v_n\in C^{\infty}([0,\infty),W^{s,k} (\mathbb{T})), \, \, \forall\, s\geq 0,\,\, \forall\, k\in[1,\infty].$$
 Replacing $(t,x)$ by $(t',x')$  and then choosing $\psi(t',x')=\eta_n(t-t', x-x')$ in \eqref{dist1} and \eqref{dist2}, we get 
 $$\begin{cases}
\partial_t u_n+c\partial_x u_n=\eta_n*\big(k(u,v,x)(v-u)\big),\,\,\,\,&x\in \mathbb{T}, \,\,t>0,\\
\partial_t v_n-c\partial_x v_n=\eta_n*\big(k(u,v,x)(u-v)\big),\,\,\,\,\,&x\in \mathbb{T}, \,\,t>0.
\end{cases}.$$  
Using this equations, we compute  $$\frac{d}{dt}\mathrm{H}[u_n(t),v_n(t)]= -\int_{\mathbb{T}} \eta_n*\big(k(u,v,x)(u-v)\big)\left(\log{u_n}-\log{v_n}\right)dx,\,\,\,\, \,\,t>0.$$
 We integrate this equation on the interval $[s,t],\, s,t>0$
 \begin{equation}\label{intH}
 \mathrm{H}[u_n(t),v_n(t)]-\mathrm{H}[u_n(s),v_n(s)]= -\int_{s}^{t}\int_{\mathbb{T}} \eta_n*\big(k(u,v,x)(u-v)\big)\left(\log{u_n}-\log{v_n}\right)dxdt'.
 \end{equation}
 \cite[Theorem 8.15]{Foll} shows  $u_n\to u, v_n\to v$ and $\eta_n*\big(k(u,v,x)(u-v)\big)\left(\log{u_n}-\log{v_n}\right)\to k(u,v,x)(u-v)\left(\log{u}-\log{v}\right)$ pointwise almost everywhere  as $n\to \infty.$ Moreover, $u_n$ and $v_n$ satisfy the bound in \eqref{bound}. This implies $|\log{u_n}-\log{v_n}|,$ $\frac{u_n}{m_{\infty}}\log{\left(\frac{u_n}{m_{\infty}}\right)}-\frac{u_n}{m_{\infty}}+1,$ and $\frac{v_n}{m_{\infty}}\log{\left(\frac{v_n}{m_{\infty}}\right)}-\frac{v_n}{m_{\infty}}+1$ are uniformly bounded. The assumptions on $k$  show $|\eta_n*\big(k(u,v,x)(u-v)\big)|\leq 2\kappa_2 M.  $  The Lebesgue dominated convergence theorem lets us  pass to the limit in \eqref{intH}  $$\mathrm{H}[u(t),v(t)]-\mathrm{H}[u(s),v(s)]= -\int_{s}^{t}\int_{\mathbb{T}} k(u,v,x)(u-v)\left(\log{u}-\log{v}\right) dxdt'.$$
 This  shows that $\mathrm{H}[u(t),v(t)]$ is absolutely continuous, hence  it is  differentiable for a.e. $t> 0$ and the claimed equality holds.\\
 $(ii)$ Let $\varepsilon$ be a sufficiently small positive number. Let $u_{\varepsilon}, v_{\varepsilon}$ be the weak solution with the initial data $u_{0,\varepsilon}\colonequals \max\{u_0,\varepsilon\}, v_{0,\varepsilon}\colonequals\max\{ v_0,\varepsilon\}. $
 Then from the first part of the proof we have 
 $$\mathrm{H}[u_{\varepsilon}(t),v_{\varepsilon}(T)]-\mathrm{H}[u_{0,\varepsilon},v_{0,\varepsilon}]= -\int_{0}^{T}\int_{\mathbb{T}} k(u_{\varepsilon},v_{\varepsilon},x)(u_{\varepsilon}-v_{\varepsilon})\left(\log{u_{\varepsilon}}-\log{v_{\varepsilon}}\right) dxdt,\,\,\,\,T>0.$$
 By the mean value theorem, there exists $\theta \in [0,1]$ such that 
$$\left(\log{u_{\varepsilon}}-\log{v_{\varepsilon}}\right)(u_{\varepsilon}-v_{\varepsilon})=\frac{(u_{\varepsilon}-v_{\varepsilon})^2}{\theta u_{\varepsilon}+(1-\theta)v_{\varepsilon}}\geq \frac{(u_{\varepsilon}-v_{\varepsilon})^2}{\max\{u_{\varepsilon},v_{\varepsilon}\}}. $$ 
Using the lower bound on $k(u_{\varepsilon},v_{\varepsilon},x),$ we can estimate 
$$\frac{k(u_{\varepsilon},v_{\varepsilon},x)}{\max\{u_{\varepsilon},v_{\varepsilon}\}}\geq \frac{k_1(x)(u_{\varepsilon}+v_{\varepsilon})^{\alpha}}{\max\{u_{\varepsilon},v_{\varepsilon}\}}\geq k_1(x)(u_{\varepsilon}+v_{\varepsilon})^{\alpha-1}\geq (2M)^{\alpha-1} \mathrm{ess\,inf}_{x\in \mathbb{T}}k_1(x)>0.$$
These estimates yield
\begin{equation}\label{H_eps}
\mathrm{H}[u_{\varepsilon}(T),v_{\varepsilon}(T)]-\mathrm{H}[u_{0,\varepsilon},v_{0,\varepsilon}]\leq  - (2M)^{\alpha-1} \mathrm{ess\,inf}_{x\in \mathbb{T}}k_1(x)\int_{0}^{T}\int_{\mathbb{T}} (u_{\varepsilon}-v_{\varepsilon})^2 dxdt.
\end{equation}
 Since $u_{\varepsilon}$ and $ v_{\varepsilon}$ are uniformly  bounded by $M,$  they converge weakly in $L^2([0,T]\times \mathbb{T})$ (or in $L^2(\mathbb{T})$ for every fixed $t\in [0,T]$) to some functions $u,v$ as $\varepsilon\to 0.$ By writing the equations \eqref{dist1} and \eqref{dist2} for $u_{\varepsilon}$ and $v_{\varepsilon},$ and letting $\varepsilon\to 0,$ we get  the limit $u,v$ is the weak solution with the initial data $u_0, v_0.$ It is easy to check that $\lim_{\varepsilon\to 0}\mathrm{H}[u_{0,\varepsilon},v_{0,\varepsilon}]=\mathrm{H}[u_{0},v_{0}].$ Also,  $\mathrm{H}[u_{\varepsilon}(t),v_{\varepsilon}(T)]$ and the $L^2$ norm are lower semi-continuous  w.r.t weak convergence (see \cite[Proposition 7.7]{San}). Hence, by letting $\varepsilon \to 0$ in \eqref{H_eps}, we get
 $$\mathrm{H}[u(T),v(T)]-\mathrm{H}[u_{0},v_{0}]\leq  - (2M)^{\alpha-1} \mathrm{ess\,inf}_{x\in \mathbb{T}}k_1(x)\int_{0}^{T}\int_{\mathbb{T}} (u-v)^2 dxdt.$$ 
\end{proof}
Let $u,v$ be any non-negative functions in $ L^{\infty}(\mathbb{T}).$ We recall the mass density $\rho\colonequals u+v$ and the momentum $j\colonequals u-v.$
Let $m_{\infty}\colonequals \frac{1}{2}\int_{\mathbb{T}}\rho dx>0$ and  $\phi$ be the solution of the Poisson equation
\begin{equation}\label{Pois}
-\partial^2_{xx}\phi =\rho-2m_{\infty},\,\,\,\,\,\, x\in \mathbb{T}. 
\end{equation}
Since $\rho-2m_{\infty}$ has zero mass, the classical elliptic regularity results show that there exists a unique solution $\phi$  satisfying 
\begin{equation}\label{E.R}
||\phi||_{H^{2}(\mathbb{T})}\leq C_R||\rho-2m_{\infty}||_{L^2(\mathbb{T})}
\end{equation}
  for some $C_R>0.$ 
With the above notations, for non-negative function $u,v\in L^{\infty}, $ we define 
\begin{equation}\label{E}
\mathrm{E}[u,v]\colonequals \mathrm{H}[u,v]+\varepsilon \int_{\mathbb{T}}j\partial_x \phi dx, 
\end{equation}
where $\varepsilon>0$ will be fixed later.
\begin{lemma}\label{Eeqv}
Let $0 \leq u, v\in L^{\infty}(\mathbb{T}),$ $M\colonequals \max\{||u||_{L^{\infty}},\,||v||_{L^{\infty}}\},$ and $m_{\infty}\colonequals \frac{1}{2}\int_{\mathbb{T}}\rho dx>0.$ If $\varepsilon>0$ is small enough,  there exist positive constants  $C_1=C_1(m_{\infty},\varepsilon)$ and $C_2=C_2(M,\varepsilon)$ such that 
\begin{equation}\label{M/mE}
 \frac{1}{C_2}\int_{\mathbb{T}}\big(({u}-{m_{\infty}})^2+({v}-{m_{\infty}})^2\big) dx\leq \mathrm{E}[u,v]\leq\frac{1}{C_1} \int_{\mathbb{T}}\big(({u}-{m_{\infty}})^2+({v}-{m_{\infty}})^2\big) dx.
\end{equation}
\end{lemma}
\begin{proof}
The H\"older inequality and \eqref{E.R} provide
\begin{align*}
\left|\int_{\mathbb{T}}j\partial_x \phi dx\right|& \leq \sqrt{\int_{\mathbb{T}}j^2 dx}\sqrt{\int_{\mathbb{T}}(\partial_x \phi)^2 dx}\\
&\leq C_{R} \sqrt{\int_{\mathbb{T}}j^2 dx}\sqrt{\int_{\mathbb{T}}(\rho-2m_{\infty})^2 dx}\\
&\leq \frac{C_R}{2}\int_{\mathbb{T}}j^2 dx+\frac{C_R}{2}\int_{\mathbb{T}}(\rho-2m_{\infty})^2 dx .
\end{align*}
Using 
$j^2 +(\rho-2m_{\infty})^2= 2 (u-m_{\infty})^2+2(v-m_{\infty})^2,$ we obtain
 $$\left|\int_{\mathbb{T}}j\partial_x \phi dx\right| \leq C_R \int_{\mathbb{T}}\big((u-m_{\infty})^2+(v-m_{\infty})^2\big)dx.$$
This estimate and \eqref{M/m} show 
\begin{align*}
\frac{1-2\varepsilon C_R M }{2M}\int_{\mathbb{T}}\big(({u}-{m_{\infty}})^2+({v}-{m_{\infty}})^2\big) dx&\leq \mathrm{E}[u,v]\\
&\leq  \frac{1+ \varepsilon C_R m_{\infty}}{m_{\infty}}\int_{\mathbb{T}}\big(({u}-{m_{\infty}})^2+({v}-{m_{\infty}})^2\big) dx.
\end{align*}
 If $0<\varepsilon<\frac{1}{2MC_R},$ then  \eqref{M/mE} holds with $C_2\colonequals \frac{2M}{1-2\varepsilon C_R M }$ and $C_1\colonequals \frac{m_{\infty}}{1+ \varepsilon C_R m_{\infty}}. $
\end{proof}
\subsection*{3.2 Proofs}
In this section we prove Theorem \ref{main} and Theorem \ref{main1}. The idea is to prove
 a Gr\"onwall  inequality for the functional $\mathrm{E}[u,v]$ in \eqref{E}.  We will need the following general functional analysis fact whose  proof can be found, for instance, in \cite[Lemma 1.2, page 260]{Temam}.
  \begin{lemma}[\bf{Lions-Magenes lemma}]\label{le:partial_t rho}
  Let $V,H,V'$  be three Hilbert spaces, each space included in the following one with dense and bounded embedding, $V'$ being the dual of $V$.   Then, for $T>0$, 
  the following inclusion 
  $$
  L^2([0,T];V)\cap H^1([0,T];V')\subset C([0,T];H)
  $$
  holds true. Moreover, for any $g \in L^2([0,T];V)\cap H^1([0,T];V')$ there holds 
  $$
  t\rightarrow ||g(t)||^2_{H} \in W^{1,1}(0,T)
  $$
  and
  \begin{equation*}
  \dfrac{d}{dt}||g(t)||^2_{H}=2\langle g'(t),g(t) \rangle_{V',V}  \,\,\,\, a.e. \text{    on   }   (0,T),
  \end{equation*}
  where $||\cdot||_H$ denotes the norm in $H,$ and $\langle \cdot, \cdot \rangle_{V',V}$ denotes the action of $V'$ on $V.$
  \end{lemma}

\bigskip

\begin{proof}[\textbf{Proof of Theorem \ref{main}}]  Theorem \ref{exist} provides that $ u$ and $v$ satisfy $$\delta\leq u,v\leq M,\,\,\, \forall\,t\geq 0,\,\,\forall\, x\in \mathbb{T}.$$  Since $k(u,v,x)$ is an interaction rate of type 3, there exist $\kappa_3=\kappa_3(\delta, M)$ and $\kappa_4=\kappa_4(\delta, M)$ such that \begin{equation}\label{k}
\kappa_3\leq k(u,v,x)\leq \kappa_4,\,\,\,\forall\,t\geq 0,\,\, \forall \, x\in \mathbb{T}.
\end{equation}
We want to compute the time derivative of $\mathrm{E}[u(t),v(t)]$ 
\begin{equation}\label{dE}\frac{d}{dt}\mathrm{E}[u(t),v(t)]=\frac{d}{dt}\mathrm{H}[u(t),v(t)]+\varepsilon \frac{d}{dt} \int_{\mathbb{T}} j\partial_x \phi dx.
\end{equation}
We need to justify this time differentiation since $u,v$ is a weak solution.
Lemma \ref{a.e.H} $(i)$ provides that $\mathrm{H}[u(t),v(t)]$ is differentiable for almost every $t\in (0, \infty)$ and
\begin{align*}
\frac{d}{dt}\mathrm{H}[u(t),v(t)]= -\int_{\mathbb{T}}k(u,v,x)\left(\log{u}-\log{v}\right)(u-v)dx.
\end{align*}
By the mean value theorem there exists $\xi$ (lying on the line connecting $u$ and $v$) such that $$\left(\log{u}-\log{v}\right)(u-v)=\frac{(u-v)^2}{\xi}\geq \frac{(u-v)^2}{M}. $$  The above estimate and \eqref{k} imply 
\begin{equation}\label{dH}\frac{d}{dt}\mathrm{H}[u(t),v(t)]\leq -\frac{\kappa_3}{M}\int_{\mathbb{T}}j^2 dx.
\end{equation}
 We have $\partial_t j=-c\partial_x \rho-2k(u,v,x)j\in H^{-1}(\mathbb{T})$  and $\partial_x \phi\in H^1(\mathbb{T}).$  Also, by differentiating \eqref{Pois} w.r.t $t>0$, we get $\partial^3_{txx} \phi=-\partial_t \rho=c\partial_x j\in H^{-1}(\mathbb{T}) $, and so  $\partial^2_{tx} \phi\in L^2(\mathbb{T}).$   Lemma \ref{le:partial_t rho} with $V=H^1(\mathbb{T})$ and $H=L^2(\mathbb{T})$ shows
\begin{equation}\label{1/4} \frac{d}{dt}\int_{\mathbb{T}} j\partial_x \phi dx=\langle \partial_t j,\partial_x \phi\rangle_{H^{-1}, H^1}+\int_{\mathbb{T}} j\partial^2_{tx} \phi dx.
\end{equation}
Using $\partial_t j=-c\partial_x \rho-2k(u,v,x)j\in H^{-1}(\mathbb{T})$ we estimate 
\begin{align}\label{dj0}
\langle \partial_t j,\partial_x \phi\rangle_{H^{-1}, H^1}=&-c\langle \partial_x\rho,\partial_x \phi\rangle_{H^{-1}, H^1}-2\int_{\mathbb{T}}k(u,v,x) j\partial_x \phi dx \nonumber\\
=& -c\int_{\mathbb{T}}(\rho-2m_{\infty})^2 dx-2\int_{\mathbb{T}}k(u,v,x) j\partial_x \phi dx. 
\end{align}
This equation can be estimated using \eqref{k}, the H\"older inequality, and \eqref{E.R} 
\begin{align}\label{dj}
\langle \partial_t j,\partial_x \phi\rangle_{H^{-1}, H^1}\leq &-c\int_{\mathbb{T}}(\rho-2m_{\infty})^2 dx+2\kappa_4C_{R}\sqrt{\int_{\mathbb{T}} j^2 dx}\sqrt{\int_{\mathbb{T}} (\rho-2m_{\infty})^2 dx}\nonumber \\
\leq & -(c-\kappa_4C_R \eta)\int_{\mathbb{T}}(\rho-2m_{\infty})^2 dx+\frac{\kappa_4C_R}{\eta}\int_{\mathbb{T}} j^2 dx,
\end{align}
where $\eta>0$ will be fixed later.
Since $\partial^3_{xxt} \phi=-\partial_t \rho=c\partial_x j\in H^{-1}(\mathbb{T})$ and so  $\partial_{t} \phi\in H^1(\mathbb{T}),$ we have  
\begin{align*}
\int_{\mathbb{T}} (\partial^2_{tx} \phi)^2 dx=&-\langle \partial^3_{xxt} \phi, \partial_t \phi \rangle_{H^{-1}, H}=-c\langle \partial_x j, \partial_{t} \phi\rangle_{H^{-1}, H} \\=&c\int_{\mathbb{T}} \partial^2_{tx} \phi j dx\leq c \sqrt{\int_{\mathbb{T}} (\partial^2_{tx} \phi)^2 dx}\sqrt{\int_{\mathbb{T}} j^2dx}.
\end{align*}
We conclude $$\int_{\mathbb{T}} (\partial^2_{tx} \phi)^2 dx\leq c^2\int_{\mathbb{T}} j^2dx$$
and so 
\begin{equation}\label{dfi}
\int_{\mathbb{T}} j\partial^2_{tx} \phi dx\leq \sqrt{\int_{\mathbb{T}} (\partial^2_{tx} \phi)^2 dx}\sqrt{\int_{\mathbb{T}} j^2dx}\leq c\int_{\mathbb{T}} j^2dx.
\end{equation}
Equations \eqref{dE}-\eqref{dfi} yield
\begin{align*}
\frac{d}{dt}\mathrm{E}[u(t),v(t)]\leq -\left(\frac{\kappa_3}{M}-\frac{\varepsilon \kappa_4C_R}{\eta}-\varepsilon c\right)\int_{\mathbb{T}}j^2 dx-\varepsilon (c-\kappa_4C_R \eta)\int_{\mathbb{T}}(\rho-2m_{\infty})^2 dx.
\end{align*}
We first fix $\eta>0$ such that $c-\kappa_4C_R \eta>0$. Then we fix $\varepsilon>0$ such that $\frac{\kappa_3}{M}-\frac{\varepsilon \kappa_4C_R}{\eta}-\varepsilon c>0$ and \eqref{M/mE} holds.
Then we have 
\begin{align*}
\frac{d}{dt}\mathrm{E}[u(t),v(t)]\leq & -\min\left\{\frac{\kappa_3}{M}-\frac{\varepsilon \kappa_4C_R}{\eta}-\varepsilon c,\varepsilon (c-\kappa_4C_R \eta) \right\}\left[\int_{\mathbb{T}}j^2 dx+\int_{\mathbb{T}}(\rho-2m_{\infty})^2 dx
\right]\\=&-2\min\left\{\frac{\kappa_3}{M}-\frac{\varepsilon \kappa_4C_R}{\eta}-\varepsilon c,\varepsilon (c-\kappa_4C_R \eta) \right\}\left[\int_{\mathbb{T}}\left({u}-{m_{\infty}}\right)^2dx+\int_{\mathbb{T}}\left({v}-{m_{\infty}}\right)^2dx\right]\\
\leq & -2C_1\min\left\{\frac{\kappa_3}{M}-\frac{\varepsilon \kappa_4C_R}{\eta}-\varepsilon c,\varepsilon (c-\kappa_4C_R \eta) \right\} \mathrm{E}[u(t),v(t)],
\end{align*}
where the last inequality follows from \eqref{M/mE}.
Gr\"onwall's inequality yields $$\mathrm{E}[u(t),v(t)]\leq e^{-2\lambda t}\mathrm{E}[u_0,v_0],\,\,\,\,\,\forall \, t\geq 0$$ with $$\lambda \colonequals C_1\min\left\{\frac{\kappa_3}{M}-\frac{\varepsilon \kappa_4C_R}{\eta}-\varepsilon c,\varepsilon (c-\kappa_4C_R \eta) \right\}. $$ 
This estimate and \eqref{M/mE} provide
\begin{align*}
\int_{\mathbb{T}}\left({u(t,x)}-{m_{\infty}}\right)^2dx&+\int_{\mathbb{T}}\left({v(t,x)}-{m_{\infty}}\right)^2dx\\ \leq & \frac{C_2}{C_1}e^{-2\lambda t} \left[\int_{\mathbb{T}}\left({u_0(x)}-{m_{\infty}}\right)^2dx+\int_{\mathbb{T}}\left({v_0(x)}-{m_{\infty}}\right)^2dx\right],\,\,\,\,\forall\,t\geq 0.
\end{align*}
\end{proof}
 
 \bigskip
 
 \begin{proof}[\textbf{Proof of Theorem \ref{main1}}] Theorem \ref{exist} provides the solution $u,v$ satisfies $$0\leq u,v\leq M, \,\,\forall\,t\geq 0,\,\forall\, x\in \mathbb{T}.$$
  We consider the functional $\mathrm{E}[u(t),v(t)]$ in \eqref{E}. We have estimated $\mathrm{H}[u(t),v(t)]$ in Lemma \ref{a.e.H} $(ii)$. We want to get similar estimate for $\int_{\mathbb{T}} j\partial_x \phi dx.$ From \eqref{1/4} its time derivative is  given by
  $$\frac{d}{dt}\int_{\mathbb{T}} j\partial_x \phi dx=\langle \partial_t j, \partial_x \phi \rangle_{H^{-1},H^1}+\int_{\mathbb{T}} j\partial^2_{tx} \phi dx.$$
 As we assume $k(u,v,x)$ is an interaction rate of type 1, there exists $\kappa_1=\kappa_1(M)$ such that $k(u,v,x)\leq \kappa_1.$ Using this fact,  $\langle \partial_t j,\partial_x \phi \rangle_{H^{-1}, H^1}$ and $\int_{\mathbb{T}} j\partial^2_{tx} \phi dx$ can be  estimated as in \eqref{dj} and \eqref{dfi}, respectively. More precisely, we have 
$$\langle \partial_t j,\partial_x \phi \rangle_{H^{-1},H^1}\leq -(c-\kappa_1C_R \eta)\int_{\mathbb{T}}(\rho-2m_{\infty})^2 dx+\frac{\kappa_1C_R}{\eta}\int_{\mathbb{T}} j^2 dx$$
and 
$$\int_{\mathbb{T}} j\partial^2_{tx} \phi dx\leq  c \int_{\mathbb{T}} j^2dx.$$ 
Let $T>0.$ The above estimates imply
$$\int_{\mathbb{T}} j(T)\partial_x \phi(T) dx\leq \int_{\mathbb{T}} j(0)\partial_x \phi(0) dx-(c-\kappa_1C_R \eta)\int_0^T\int_{\mathbb{T}}(\rho-2m_{\infty})^2 dxdt+\left(\frac{\kappa_1 C_R}{\eta}+c\right)\int_0^T\int_{\mathbb{T}} j^2dxdt.$$
This estimate and Lemma \ref{a.e.H} $(ii)$ show
\begin{align*}\mathrm{E}[u(T),v(T)]\leq & \mathrm{E}[u(0),v(0)]-\varepsilon (c-\kappa_1C_R \eta)\int_0^T\int_{\mathbb{T}}(\rho-2m_{\infty})^2 dxdt\\
&-\left((2M)^{\alpha-1} \mathrm{ess\,inf}_{x\in \mathbb{T}}k_1(x)-\varepsilon\left(\frac{\kappa_1 C_R}{\eta}+c\right)\right)\int_0^T\int_{\mathbb{T}} j^2dxdt.
\end{align*}
We first fix $\eta>0$ such that $c-\kappa_1C_R \eta>0.$
Then we choose $\varepsilon>0$ such that
$$(2M)^{\alpha-1} \mathrm{ess\,inf}_{x\in \mathbb{T}}k_1(x)-\varepsilon\left(\frac{\kappa_1 C_R}{\eta}+c\right)>0$$
 and \eqref{M/mE} holds.  
We use 
$(\rho-2m_{\infty})^2+j^2=2(u-m_{\infty})^2+2(v-m_{\infty})^2$
and \eqref{M/mE} to get
$$\mathrm{E}[u(T),v(T)]\leq  \mathrm{E}[u(0),v(0)]-2\lambda_{\alpha} \int_0^T\mathrm{E}[u(t),v(t)]dt,$$
where $$
\lambda_{\alpha}\colonequals C_1 \min\left\{(2M)^{\alpha-1} \mathrm{ess\,inf}_{x\in \mathbb{T}}k_1(x)-\varepsilon\left(\frac{\kappa_1 C_R}{\eta}+c\right), \varepsilon (c-\kappa_1C_R \eta) \right\}. $$ Grönwall's inequality provides $\mathrm{E}[u(t),v(t)]\leq e^{-2\lambda_{\alpha}}\mathrm{E}[u_0,v_0]$ for all $t\geq 0.$  Then we use \eqref{M/mE}  to get
\begin{align*}
\int_{\mathbb{T}}\left({u(t,x)}-{m_{\infty}}\right)^2dx&+\int_{\mathbb{T}}\left({v(t,x)}-{m_{\infty}}\right)^2dx\\ \leq  & \frac{C_2}{C_1}e^{-2\lambda_{\alpha} t} \left[\int_{\mathbb{T}}\left({u_0(x)}-{m_{\infty}}\right)^2dx+\int_{\mathbb{T}}\left({v_0(x)}-{m_{\infty}}\right)^2dx\right],\,\,\,\forall\, t\geq 0.
\end{align*} 
 \end{proof}

\begin{Remark}\label{Rema} In Theorems \ref{main} and \ref{main1} 
 the convergence rates $\lambda$ and multiplicative constants $\Lambda$ depend on the $L^\infty$-norms of $u_0$ and $v_0.$ 
 When \eqref{1Eq} is linear (i.e., $k(u,v,x)=k_1(x)\in L^{
\infty}(\mathbb{T})$), we can only require $u_0,\,v_0\in L^2(\mathbb{T})$ and  $\mathrm{ess\,inf}_{x\in \mathbb{T} } k_1(x)>0$, and we can obtain convergence rates and multiplicative constants independent of $u_0$ and $v_0.$  
 To do that we consider 
$$\mathrm{E}_{\mathrm{lin}}[u,v]\colonequals \int_{\mathbb{T}}((u-m_{\infty})^2+(v-m_{\infty})^2)dx+\varepsilon \int_{\mathbb{T}}j\partial_x \phi dx.$$
 Similar arguments in the proof of Lemma \ref{Eeqv} yield
\begin{equation}
\label{eq.l}
(1-C_{R}\varepsilon ) \int_{\mathbb{T}}((u-m_{\infty})^2+(v-m_{\infty})^2)dx \leq \mathrm{E}_{\mathrm{lin}}[u,v]\leq (1+C_R \varepsilon)\int_{\mathbb{T}}((u-m_{\infty})^2+(v-m_{\infty})^2)dx.
\end{equation}
Hence, if $\varepsilon<1/C_R,$ then $\mathrm{E}_{\mathrm{lin}}[u,v]$ is equivalent to $\int_{\mathbb{T}}((u-m_{\infty})^2+(v-m_{\infty})^2)dx.$   We use \eqref{d/dt}, \eqref{1/4}, \eqref{dj0}, \eqref{dfi}, and the H\"older inequality 
\begin{align*}
\frac{d}{dt}\mathrm{E}_{\mathrm{lin}}[u(t),v(t)]&\leq -2\int_{\mathbb{T}}k_1(x)j^2dx-\varepsilon c\int_{\mathbb{T}}(\rho-2m_{\infty})^2dx- 2\varepsilon \int_{\mathbb{T}}k_1(x)j\partial_x \phi dx+\varepsilon c\int_{\mathbb{T}}j^2dx\\
&\leq -(2k_0-\varepsilon c)\int_{\mathbb{T}}j^2dx-\varepsilon c\int_{\mathbb{T}}(\rho-2m_{\infty})^2dx- 2\varepsilon K_0 \sqrt{\int_{\mathbb{T}}j^2dx}\sqrt{\int_{\mathbb{T}}(\partial_x \phi)^2 dx},
\end{align*}
where  $ k_0\colonequals \mathrm{ess\,inf}_{x\in \mathbb{T} } k_1(x)$ and $ K_0\colonequals \mathrm{ess\,sup}_{x\in \mathbb{T} } k_1(x)$. We use \eqref{E.R} and the H\"older inequality 
$$\frac{d}{dt}\mathrm{E}_{\mathrm{lin}}[u(t),v(t)]\leq -\left(2k_0-\varepsilon\left(\frac{C_R K_0}{\eta}+c\right)\right) \int_{\mathbb{T}}j^2dx-\varepsilon (c-C_RK_0 \eta)\int_{\mathbb{T}}(\rho-2m_{\infty})^2dx,\,\,\,\,\eta>0. $$
We choose $\varepsilon\in (0,1/C_R)$ and $\eta>0$ such that $2k_0-\varepsilon\left(\frac{C_R K_0}{\eta}+c\right)>0$ and $c-C_RK_0 \eta>0.$
\eqref{eq.l} and similar arguments in the proof of Theorem \ref{main} yield 
$$\int_{\mathbb{T}}((u(t,x)-m_{\infty})^2+(v(t,x)-m_{\infty})^2)dx\leq \frac{1+\varepsilon C_R }{1-\varepsilon C_R} e^{-2 \lambda t}\int_{\mathbb{T}}((u_0(x)-m_{\infty})^2+(v_0(x)-m_{\infty})^2)dx,$$
where
$$\lambda\colonequals \frac{1}{1+(2\pi)^{-2}\varepsilon}\min\left\{2k_0-\varepsilon\left(\frac{C_R K_0}{\eta}+c\right), \varepsilon (c-C_RK_0 \eta) \right\}.$$ To make $\lambda=\lambda(\varepsilon,\eta)$ as large as possible we choose $\eta> 0$ such that
$2k_0-\varepsilon\left(\frac{C_R K_0}{\eta}+c\right)= \varepsilon (c-C_RK_0 \eta)>0.$
This condition holds if $\eta=\sqrt{\frac{(k_0-\varepsilon c)^2}{\varepsilon^2C_R^2K_0^2}+1}-\frac{k_0-\varepsilon c}{\varepsilon C_RK_0}$ and  $\varepsilon\in \big(0,\frac{2c k_0}{c^2+C_R^2K_0^2}\big).$ With this choice of $\eta$ we have  $$\lambda=\lambda(\varepsilon)=
k_0\frac{1-\sqrt{(1-\frac{\varepsilon c}{k_0})^2+\frac{\varepsilon^2C_R^2K_0^2}{k_0^2}}}{1+\varepsilon C_R}.$$  One can check that $\lambda=\lambda(\varepsilon)$ attains its maximum at  
$\varepsilon=4\pi^2
\frac{a^3+2a^2+b^2(a+1)-b\sqrt{a^2+2a+b^2}}{(a^2+b^2)(a^2+2a +b^2)-b^2}$, where  $a\colonequals \frac{c}{C_Rk_0}$ and $b\colonequals \frac{K_0^2}{k_0^2}. $
The maximum value is 
$$\lambda_{\mathrm{max}}\colonequals k_0\frac{(a^2+b^2)\sqrt{a^2+2a+b^2}-b(a^2+2a+b^2)+ab}{(a^2+a+b^2)\sqrt{a^2+2a+b^2}-b}.$$
The following example shows that the convergence rate $\lambda_{\mathrm{max}}$ is not optimal: Let $c=1$ and $$k_1(x)\colonequals \begin{cases}\pi,\,\,\,\, &0<x\leq 1/2\\
4\pi,\,\,\,&1/2<x\leq 1\end{cases}.$$ Then we have $k_0=\mathrm{ess\,inf}_{x\in \mathbb{T} } k_1(x)=\pi$, $K_0=\mathrm{ess\,sup}_{x\in \mathbb{T} } k_1(x)=4\pi.$ The $L^2$-norm of $\partial_x \phi$ can be  estimated as 
$||\partial_x \phi||_{L^2(\mathbb{T})}\leq (2\pi)^{-2}|| \rho-m_{\infty}||_{L^2(\mathbb{T})},$ hence we can take $C_R=(2\pi)^{-2}.$    This let us compute  $\lambda_{\mathrm{max}}\approx 0.20731\cdot\pi$. It is  less than the optimal rate $0.86845\cdot\pi$ which was computed in \cite[page 32]{Ant}.
 \end{Remark}
\section{Three dimensional model}
 
    As we showed in \eqref{d/dt3}, if $u_i,v_i,$ $i\in\{1,2,3\}$ is a classical solution to \eqref{2Eq}, then, for any convex function $\varphi,$ the integral $\int_{\mathbb{T}^3}\sum_{i=1}^3 (\varphi (u_i)+\varphi(v_i))dx$
 is non-increasing in $t\geq 0.$  We consider again the case $\varphi (u)=u\log{\left(\frac{u}{m_{\infty}}\right)}-{u}+{m_{\infty}}$ with continuous extension to 0. We introduce Boltzmann's entropy for the three dimensional equation \eqref{2Eq} as the sum
\begin{equation*}
\sum_{i=1}^3\mathrm{H}[u_i,v_i]\colonequals \int_{\mathbb{T}^3}\sum_{i=1}^3\left[\left(\frac{u_i}{m_{\infty}}\log{\left(\frac{u_i}{m_{\infty}}\right)}-\frac{u_i}{m_{\infty}}+1\right)m_{\infty}+\left(\frac{v_i}{m_{\infty}}\log{\left(\frac{v_i}{m_{\infty}}\right)}-\frac{v_i}{m_{\infty}}+1\right)m_{\infty}\right]dx
\end{equation*}
defined for non-negative functions $u_i,v_i \in L^{\infty}( \mathbb{T}^3),i\in\{1,2,3\} $ and  $m_{\infty}\colonequals \frac{1}{6}\int_{\mathbb{T}^3}\sum_{i=1}^3(u_i+v_i)dx>0.$  
\begin{lemma}\label{a.e.H3}
$(i)$ Let $k(u_1,u_2,u_3,v_1,v_2,v_3,x)$ be an interaction rate of type 2.  Let $u_i,v_i\in L^{\infty}([0,\infty)\times \mathbb{T}^3),$ $i\in \{1,2,3\}$ be the weak solution of \eqref{1Eq} with the initial data 
$\delta \leq u_{0,i}, v_{0,i}\in L^{\infty}(\mathbb{T}^3)$ for some $\delta>0$. Let  $m_{\infty}\colonequals\frac{1}{6}\int_{\mathbb{T}}\sum_{i=1}^3(u_{0,i}+v_{0,i})dx>0.$ Then, for almost every $t\in (0,\infty),$ we have
\begin{align}\label{dtH3}
\frac{d}{dt}\sum_{i=1}^3 \mathrm{H}[u_i,v_i]=& -\frac{1}{2}\int_{\mathbb{T}^3}k\sum_{\substack{i,j=1\\i\neq j}}^3(\log{u_i}-\log{u_j})(u_i-u_j) dx\nonumber\\
 &-\frac{1}{2}\int_{\mathbb{T}^3}k\sum_{\substack{i,j=1\\i\neq j}}^3(\log{v_i}-\log{v_j})(v_i-v_j) dx\nonumber \\
 &- \int_{\mathbb{T}^3}k\sum_{i, j=1}^3(\log{u_i}-\log{v_j})(u_i-v_j) dx\leq 0. 
\end{align}
$(ii)$ Let $k(u_1,u_2,u_3,v_1,v_2,v_3,x)$ be an interaction rate of type 1. Assume $$k(u_1,u_2,u_3,v_1,v_2,v_3,x)\geq k_1(x)(u_1+u_2+u_3+v_1+v_2+v_3)^{\alpha}$$ holds  for some $\alpha\in [0,1]$ and $k_1\in L^{\infty}(\mathbb{T}^3)$ with $\mathrm{ess\,inf}_{x\in \mathbb{T}^3 } k_1(x)>0.$   Let $u_i,v_i\in L^{\infty}([0,\infty)\times \mathbb{T}),$ $i\in \{1,2,3\}$ be the weak solution of \eqref{1Eq} with the initial data  
$0 \leq u_{0,i}, v_{0,i}\in L^{\infty}(\mathbb{T}^3).$  Let  $M\colonequals \max\{||u_0||_{L^{\infty}},\,||v_0||_{L^{\infty}}\}$ and  $m_{\infty}\colonequals\frac{1}{6}\int_{\mathbb{T}^3}\sum_{i=1}^3(u_{0,i}+v_{0,i})dx>0.$ Then, for any $T>0,$ we have
\begin{align*}
&\sum_{i=1}^3 \mathrm{H}[u_i(T),v_i(T)]-\sum_{i=1}^3 \mathrm{H}[u_i(0),v_i(0)]\\ &\leq  -\frac{(6M)^{\alpha-1}\mathrm{ess\,inf}_{x\in \mathbb{T}^3}k_1(x)}{2}\int_0^T\int_{\mathbb{T}^3}\Big[\sum_{\substack{i,j=1\\i\neq j}}^3\big((u_i-u_j)^2 +(v_i-v_j)^2\big) + 2\sum_{i, j=1}^3(u_i-v_j)^2 \Big]dxdt. 
\end{align*}
\end{lemma}
\begin{proof}
$(i)$ The proof follows from the same arguments as   the one dimensional case in Lemma \ref{a.e.H} $(i)$, hence we omit the details. \\
$(ii)$ As in the proof of Lemma \ref{a.e.H} $(ii),$ for $\varepsilon>0,$ we consider the weak solution $u_{i,\varepsilon}, v_{i,\varepsilon},$ $i\in \{1,2,3\}$ with the initial data  $u_{0,i,\varepsilon}\colonequals \max\{u_{0,i}, \varepsilon\}, v_{0,i,\varepsilon} \colonequals \max\{v_{0,i}, \varepsilon\}.$
Then \eqref{dtH3} holds for  $u_{i,\varepsilon}, v_{i,\varepsilon},$ i.e.,
\begin{align*}
\frac{d}{dt}\sum_{i=1}^3 \mathrm{H}[u_{i,\varepsilon},v_{i,\varepsilon}]=& -\frac{1}{2}\int_{\mathbb{T}^3}k\sum_{\substack{i,j=1\\i\neq j}}^3(\log{u_{i,\varepsilon}}-\log{u_{j,\varepsilon}})(u_{i,\varepsilon}-u_{j,\varepsilon}) dx\nonumber\\
 &-\frac{1}{2}\int_{\mathbb{T}^3}k\sum_{\substack{i,j=1\\i\neq j}}^3(\log{v_{i,\varepsilon}}-\log{v_{j,\varepsilon}})(v_{i,\varepsilon}-v_{j,\varepsilon}) dx\nonumber \\
 &- \int_{\mathbb{T}^3}k\sum_{i, j=1}^3(\log{u_{i,\varepsilon}}-\log{v_{j,\varepsilon}})(u_{i,\varepsilon}-v_{j,\varepsilon}) dx\leq 0. 
\end{align*}
The mean value theorem  provides that there exists $\theta\in [0,1]$ such that   \begin{align*}
\frac{d}{dt}\sum_{i=1}^3 \mathrm{H}[u_{i,\varepsilon},v_{i,\varepsilon}]=& -\frac{1}{2}\int_{\mathbb{T}^3} k\sum_{\substack{i,j=1\\i\neq j}}^3\frac{(u_{i,\varepsilon}-u_{j,\varepsilon})^2}{\theta u_{i,\varepsilon}+(1-\theta)u_{j,\varepsilon}} dx
 -\frac{1}{2}\int_{\mathbb{T}^3}k\sum_{\substack{i,j=1\\i\neq j}}^3\frac{(v_{i,\varepsilon}-v_{j,\varepsilon})^2}{\theta v_{i,\varepsilon}+(1-\theta)v_{j,\varepsilon}} dx\nonumber \\
 &- \int_{\mathbb{T}^3}k\sum_{i, j=1}^3\frac{(u_{i,\varepsilon}-v_{j,\varepsilon})^2}{\theta u_{i,\varepsilon}+(1-\theta)v_{j,\varepsilon}} dx\\
 \leq & -\frac{1}{2}\int_{\mathbb{T}^3}k\sum_{\substack{i,j=1\\i\neq j}}^3\frac{(u_{i,\varepsilon}-u_{j,\varepsilon})^2}{\max\{u_{i,\varepsilon},u_{j,\varepsilon}\}} dx-\frac{1}{2}\int_{\mathbb{T}^3}k\sum_{\substack{i,j=1\\i\neq j}}^3\frac{(v_{i,\varepsilon}-v_{j,\varepsilon})^2}{\max\{v_{i,\varepsilon},v_{j,\varepsilon}\}} dx\nonumber \\
 &- \int_{\mathbb{T}^3}k\sum_{i, j=1}^3\frac{(u_{i,\varepsilon}-v_{j,\varepsilon})^2}{\max\{u_{i,\varepsilon},v_{j,\varepsilon}\}} dx. 
\end{align*}
Because of the lower bound on $k,$ the terms  $$\frac{k}{\max\{u_{i,\varepsilon},u_{j,\varepsilon}\}},\,\,\frac{k}{\max\{v_{i,\varepsilon},v_{j,\varepsilon}\}},\,\, \frac{k}{\max\{u_{i,\varepsilon},v_{j,\varepsilon}\}}$$
are bigger than $(u_{1,\varepsilon}+u_{2,\varepsilon}+u_{3,\varepsilon}+v_{1,\varepsilon}+v_{2,\varepsilon}+v_{3,\varepsilon})^{\alpha-1} \mathrm{ess\,inf}_{x\in \mathbb{T}^3}k_1(x).$ Then,  the upper bounds on the solution and  $\alpha\in [0,1]$ imply  
$$(u_{1,\varepsilon}+u_{2,\varepsilon}+u_{3,\varepsilon}+v_{1,\varepsilon}+v_{2,\varepsilon}+v_{3,\varepsilon})^{\alpha-1} \mathrm{ess\,inf}_{x\in \mathbb{T}^3}k_1(x)\geq (6M)^{\alpha-1}\mathrm{ess\,inf}_{x\in \mathbb{T}^3}k_1(x).$$
The above estimates yield 
\begin{align*}
\frac{d}{dt}&\sum_{i=1}^3 \mathrm{H}[u_{i,\varepsilon},v_{i,\varepsilon}]\\ \leq & -\frac{(6M)^{\alpha-1}\mathrm{ess\,inf}_{x\in \mathbb{T}^3}k_1(x)}{2}\int_{\mathbb{T}^3}\Big[\sum_{\substack{i,j=1\\i\neq j}}^3\big((u_{i,\varepsilon}-u_{j,\varepsilon})^2 +(v_{i,\varepsilon}-v_{j,\varepsilon})^2\big) + 2\sum_{i, j=1}^3(u_{i,\varepsilon}-v_{j,\varepsilon})^2 \Big]dx. 
\end{align*}
We integrate  this estimate on $[0,T]$ 
\begin{align*}
&\sum_{i=1}^3 \mathrm{H}[u_{i,\varepsilon}(T),v_{i,\varepsilon}(T)]-\sum_{i=1}^3 \mathrm{H}[u_{i,\varepsilon}(0),v_{i,\varepsilon}(0)]\\ \leq & -\frac{(6M)^{\alpha-1}\mathrm{ess\,inf}_{x\in \mathbb{T}^3}k_1(x)}{2}\int_0^T\int_{\mathbb{T}^3}\Big[\sum_{\substack{i,j=1\\i\neq j}}^3\big((u_{i,\varepsilon}-u_{j,\varepsilon})^2 +(v_{i,\varepsilon}-v_{j,\varepsilon})^2\big) + 2\sum_{i, j=1}^3(u_{i,\varepsilon}-v_{j,\varepsilon})^2 \Big]dxdt. 
\end{align*}
As in the proof of Lemma \ref{a.e.H} $(ii)$,  $u_{i,\varepsilon}$ and $v_{j,\varepsilon}$ are uniformly  bounded by $M$ for all $i\in\{1,2,3\},$ hence they converge weakly in $L^2([0,T]\times \mathbb{T}^3)$ (or in $L^2( \mathbb{T}^3)$ for every fixed $t\in [0,T]$) to the weak solution $u_{i},v_{j}$ with the initial data $u_{0,i},v_{0,j}.$ The claimed estimate follows by letting $\varepsilon \to 0 $ in the above estimate and using  the lower semi-continuity of  $ \mathrm{H}$ and the $L^2$ norm w.r.t weak convergence. 
\end{proof}
Let $u_i,v_i,$ $i\in \{1,2,3\}$ be non-negative functions in $ L^{\infty}(\mathbb{T}^3).$ We recall the mass density
$$\rho\colonequals \sum_{i=1}^3 (u_i+v_i)$$ and the momentum $$j\colonequals \begin{pmatrix}
u_1-v_1\\
u_2-v_2\\
u_3-v_3
\end{pmatrix} \in \mathbb{R}^3.$$
Let $m_{\infty}\colonequals \frac{1}{6}\int_{\mathbb{T}^3}\rho dx>0$ and  $\phi$ be the solution of the Poisson equation
\begin{equation}\label{Pois3}
-\Delta_{x}\phi =\rho-6m_{\infty},\,\,\,\,\,\, x\in \mathbb{T}^3.
\end{equation}
Similar to \eqref{E.R}, there exists a unique solution $\phi$ satisfying 
\begin{equation}\label{E.R3}
||\phi||_{H^2(\mathbb{T}^3)}\leq C_R ||\rho-6m_{\infty}||_{L^2(\mathbb{T}^3)}
\end{equation}
for some $C_R>0.$
With the above notations, for non-negative function $u_i,v_i\in  L^{\infty}(\mathbb{T}^3),\, i\in \{1,2,3\},$   we define 
\begin{equation}\label{E3}
\mathrm{E}[u_1, u_2, u_3,v_1,v_2,v_3]\colonequals \sum_{i=1}^3\mathrm{H}[u_i,v_i]+\varepsilon \int_{\mathbb{T}^3}j\cdot \nabla_x \phi dx, 
\end{equation}
where $\varepsilon>0$  will be fixed later. 

Following the proofs of Lemma \ref{Heqv} and Lemma \ref{Eeqv} we can prove the following lemma; we omit its proof.
\begin{lemma}
Let $0 \leq u_{i}, v_i\in L^{\infty}(\mathbb{T}^3)$ for all $i\in \{1,2,3\},$  $\displaystyle M\colonequals \max_{i\in \{1,2,3\}}\{||u_i||_{L^{\infty}},\,||v_i||_{L^{\infty}}\},$ and $m_{\infty}\colonequals \frac{1}{6}\int_{\mathbb{T}^3}\rho dx>0$. If $\varepsilon>0$ is small enough,  there exist positive constants  $C_1=C_1(m_{\infty},\varepsilon)$ and $C_2=C_2(M,\varepsilon)$ such that 
\begin{align}\label{M/mE3}
  \frac{1}{C_2}\int_{\mathbb{T}^3} \sum_{i=1}^3\left[\left({u_i}-{m_{\infty}}\right)^2+\left({v_i}-{m_{\infty}}\right)^2\right]dx &\leq \mathrm{E}[u_1, u_2, u_3,v_1,v_2,v_3]\nonumber \\
  &\leq  \frac{1}{C_1}\int_{\mathbb{T}^3} \sum_{i=1}^3\left[\left({u_i}-{m_{\infty}}\right)^2+\left({v_i}-{m_{\infty}}\right)^2\right]dx.
\end{align}
\end{lemma}

\bigskip

\begin{proof}[\textbf{Proof of Theorem \ref{main3}}]  Theorem \ref{exist3} provides that $ u_i$ and $v_i$ satisfy $$\delta\leq u_i,v_i\leq M,\,\,\, \forall\, t\geq 0,\,\,\forall\,x\in \mathbb{T}^3.$$ Since $k(u_1,...,v_3,x)$ is an interaction rate  of type 3, there exist $\kappa_3=\kappa_3(\delta, M)$ and $\kappa_4=\kappa_4(\delta, M)$ such that \begin{equation}\label{kappa}\kappa_3\leq k(u_1,...,v_3,x)\leq \kappa_4 ,\,\,\, \forall\, t\geq 0,\,\,\forall\,x\in \mathbb{T}^3.
\end{equation}
We have 
\begin{equation}\label{dE3}\frac{d}{dt}\mathrm{E}[u_1(t),...,v_3(t)]=\frac{d}{dt}\sum_{i=1}^3\mathrm{H}[u_i(t),v_i(t)]+\varepsilon \frac{d}{dt} \int_{\mathbb{T}} j\cdot \nabla_x \phi dx.
\end{equation}
Applying the mean value theorem to the right hand side of \eqref{dtH3} and doing the same arguments to obtain \eqref{dH},  we can show
\begin{align}\label{dtH3<}
\frac{d}{dt}\sum_{i=1}^3 \mathrm{H}[u_i,v_i]\leq & -\frac{\kappa_3}{2M}\int_{\mathbb{T}^3}\sum_{\substack{i,j=1\\i\neq j}}^3(u_i-u_j)^2 dx-\frac{\kappa_3}{2M} \int_{\mathbb{T}^3}\sum_{\substack{i,j=1\\i\neq j}}^3(v_i-v_j)^2 dx\nonumber \\
 &-\frac{\kappa_3}{M} \int_{\mathbb{T}^3}\sum_{i, j=1}^3(u_i-v_j)^2 dx\leq 0. 
\end{align}
The second equation in \eqref{2Eq0} can be written as \begin{equation}\label{j/3}\partial_t j=-\frac{c}{3}\nabla_x \rho-6kj+\frac{c}{3}\begin{pmatrix}
\partial_{x_1}(u_2+u_3+v_2+v_3-2u_1-2v_1)\\
\partial_{x_2}(u_1+u_3+v_1+v_3-2u_2-2v_2)\\
\partial_{x_3}(u_1+u_2+v_1+v_2-2u_3-2v_3)
\end{pmatrix}\in [H^{-1}(\mathbb{T}^3)]^3.
\end{equation} 
By differentiating \eqref{Pois3} w.r.t $t>0,$ we get  $-\Delta_x(\partial_t \phi)=-\partial_t \rho=c\mathrm{div}_x j\in H^{-1}(\mathbb{T}^3)$ and so $\nabla_x (\partial_t \phi)\in [H^1(\mathbb{T}^3)]^3.$
Lemma \ref{le:partial_t rho} with $V\colonequals H^1(\mathbb{T}^3)$ and $H\colonequals L^2(\mathbb{T}^3)$ provides \begin{equation}\label{djphi+}\frac{d}{dt} \int_{\mathbb{T}} j\cdot \nabla_x \phi dx=\langle\partial_t j, \nabla_x \phi\rangle_{H^{-1},H^1}+\int_{\mathbb{T}} j\cdot \nabla_x (\partial_t \phi) dx.
\end{equation}
We use \eqref{j/3} to compute 
\begin{align}\label{dtjphi}
\langle \partial_t j, \nabla_x \phi\rangle_{H^{-1}, H^1} =&-\frac{c}{3}\langle \nabla_x\rho , \nabla_x \phi \rangle_{H^{-1}, H^1} -6\int_{\mathbb{T}^3}k j\cdot \nabla_x \phi dx\nonumber\\
&+\frac{c}{3}  \langle\partial_{x_1}(u_2+u_3+v_2+v_3-2u_1-2v_1), \partial_{x_1}\phi \rangle_{H^{-1}, H^1} \nonumber\\&+\frac{c}{3}\langle  \partial_{x_2}(u_1+u_3+v_1+v_3-2u_2-2v_2), \partial_{x_2} \phi \rangle_{H^{-1}, H^1}\nonumber\\
&+\frac{c}{3}\langle \partial_{x_3}(u_1+u_2+v_1+v_2-2u_3-2v_3) , \partial_{x_3} \phi\rangle_{H^{-1}, H^1} \nonumber \\
=&-\frac{c}{3}\int_{\mathbb{T}^3}(\rho-6m_{\infty})^2 dx-6\int_{\mathbb{T}^3}k j\cdot \nabla_x \phi dx \nonumber \\
&-\frac{c}{3} \int_{\mathbb{T}^3}\partial^2_{x_1x_1}\phi (u_2+u_3+v_2+v_3-2u_1-2v_1) dx \nonumber \\&-\frac{c}{3}\int_{\mathbb{T}^3}\partial^2_{x_2x_2} \phi (u_1+u_3+v_1+v_3-2u_2-2v_2) dx \nonumber \\
&-\frac{c}{3}\int_{\mathbb{T}^3}\partial^2_{x_3x_3} \phi (u_1+u_2+v_1+v_2-2u_3-2v_3) dx.
\end{align}
The Cauchy-Schwarz inequality shows 
\begin{align}\label{U}
\partial^2_{x_1x_1}\phi (u_2+u_3+v_2+v_3-2u_1-2v_1) &+\partial^2_{x_2x_2} \phi (u_1+u_3+v_1+v_3-2u_2-2v_2)\nonumber \\ &+\partial^2_{x_3x_3} \phi (u_1+u_2+v_1+v_2-2u_3-2v_3) \nonumber \\
\leq&  \sqrt{(\partial^2_{x_1x_1}\phi)^2+(\partial^2_{x_2x_2}\phi)^2+(\partial^2_{x_3x_3}\phi)^2}\sqrt{U},
\end{align}
where  $$U\colonequals (u_2+u_3+v_2+v_3-2u_1-2v_1)^2+(u_1+u_3+v_1+v_3-2u_2-2v_2)^2 +(u_1+u_2+v_1+v_2-2u_3-2v_3)^2.$$
 $U$ can be estimated as
\begin{equation}\label{U1}
U\leq 8[(u_1-v_2)^2+(u_1-v_3)^2+(u_2-v_1)^2+(u_2-v_3)^2+(u_3-v_1)^2+(u_3-v_2)^2]=8\sum_{\substack{i,j=1\\i\neq j}}^3(u_i-v_j)^2.
\end{equation}
\eqref{dtjphi}, \eqref{U},  \eqref{U1}, and \eqref{kappa} imply 
\begin{align*}
\langle \partial_t j, \nabla_x \phi \rangle_{H^{-1}, H^1} 
\leq & -\frac{c}{3}\int_{\mathbb{T}^3}(\rho-6m_{\infty})^2 dx+6\kappa_4 \sqrt{\int_{\mathbb{T}^3} |j|^2dx}\sqrt{\int_{\mathbb{T}^3 }|\nabla_x \phi|^2 dx}\\
&+\frac{2\sqrt{2}c}{3}\sqrt{\int_{\mathbb{T}^3}\big((\partial^2_{x_1x_1}\phi)^2+(\partial^2_{x_2x_2}\phi)^2+(\partial^2_{x_3x_3}\phi)^2\big)dx}\sqrt{\int_{\mathbb{T}^3}\sum_{\substack{i,j=1\\i\neq j}}^3(u_i-v_j)^2dx}.
\end{align*}
 Because of the elliptic regularity in \eqref{E.R3}, we have
\begin{align*}
\langle \partial_t j, \nabla_x \phi \rangle_{H^{-1}, H^1}
\leq & -\frac{c}{3}\int_{\mathbb{T}^3}(\rho-6m_{\infty})^2 dx+6\kappa_4C_{R} \sqrt{\int_{\mathbb{T}^3} |j|^2dx}\sqrt{\int_{\mathbb{T}^3 }(\rho-6m_{\infty})^2 dx}\\
&+ \frac{2\sqrt{2}cC_{R}}{3}\sqrt{\int_{\mathbb{T}^3 }(\rho-6m_{\infty})^2 dx}\sqrt{\int_{\mathbb{T}^3}\sum_{i\neq j}^3(u_i-v_j)^2dx}.
\end{align*}
Using the elementary inequalities $$|j|^2\leq \sum_{i, j=1}^3(u_i-v_j)^2,\,\,\,\, \sum_{\substack{i,j=1\\i\neq j}}^3(u_i-v_j)^2\leq \sum_{i, j=1}^3(u_i-v_j)^2,$$ 
 and the H\"older inequality, we get
\begin{align}\label{dtj<sum}
\langle \partial_t j, \nabla_x \phi \rangle_{H^{-1}, H^1}
\leq & -\left(\frac{c}{3}-\frac{\eta cC_{R}}{3}-3\eta \kappa_4C_{R}\right)\int_{\mathbb{T}^3}(\rho-6m_{\infty})^2 dx\nonumber\\&+\left(\frac{3\kappa_4 C_{R}}{\eta}+\frac{2cC_{R}}{3 \eta}\right) \int_{\mathbb{T}^3}\sum_{i,j=1}^3(u_i-v_j)^2dx,
\end{align}
where $\eta>0$ will be fixed later.

Next, we use $\Delta_x (\partial_t \phi)=-\partial_t \rho=c\mathrm{div}_x j\in H^{-1}(\mathbb{T}^3)$ and $\partial_t \phi\in H^{1}(\mathbb{T}^3)$ to estimate 
\begin{align*}
\int_{\mathbb{T}^3} |\nabla_x(\partial_{t} \phi)|^2 dx&=-\langle \Delta_x(\partial_{t} \phi),\partial_t \phi \rangle_{H^{-1}, H^1} =-c  \langle\mathrm{div}_x j, \partial_{t} \phi \rangle_{H^{-1}, H^1}\\ & =c\int_{\mathbb{T}^3} \nabla_x(\partial_{t} \phi) \cdot j dx\leq c \sqrt{\int_{\mathbb{T}^3} |\nabla_x (\partial_{t} \phi)|^2 dx}\sqrt{\int_{\mathbb{T}} |j|^2dx}.
\end{align*}
We conclude $$\int_{\mathbb{T}^3} |\nabla_x (\partial_{t} \phi)|^2 dx\leq c^2\int_{\mathbb{T}^3} |j|^2dx$$
and so
\begin{equation}\label{dfi3}
\int_{\mathbb{T}^3} \nabla_x(\partial_{t} \phi) \cdot j dx\leq \sqrt{\int_{\mathbb{T}^3} |\nabla_x(\partial_{t} \phi)|^2 dx}\sqrt{\int_{\mathbb{T}^3} |j|^2dx}\leq c\int_{\mathbb{T}^3} |j|^2dx\leq c \int_{\mathbb{T}^3}\sum_{i,j=1}^3(u_i-v_j)^2dx.
\end{equation}
\eqref{dE3}, \eqref{dtH3<}, \eqref{dtj<sum}, and  \eqref{dfi3} yield
\begin{align}\label{dtEep}
\frac{d}{dt}\mathrm{E}[u_1(t),...,v_3(t)] \leq & -\frac{\kappa_3}{2M}\int_{\mathbb{T}^3}\sum_{\substack{i,j=1\\i\neq j}}^3(u_i-u_j)^2 dx-\frac{\kappa_3}{2M}\int_{\mathbb{T}^3}\sum_{\substack{i,j=1\\i\neq j}}^3(v_i-v_j)^2 dx\nonumber \\
 &-\left(\frac{\kappa_3}{M} -\varepsilon \left(\frac{3\kappa_4C_{R}}{\eta}+\frac{2cC_{R}}{3 \eta}+c\right)\right)\int_{\mathbb{T}^3}\sum_{i, j=1}^3(u_i-v_j)^2 dx\nonumber \\&-\varepsilon\left(\frac{c}{3}-\frac{\eta cC_{R}}{3}-3\eta \kappa_4 C_{R}\right)\int_{\mathbb{T}^3}(\rho-6m_{\infty})^2 dx. 
\end{align}
We first fix $\eta>0$ such that
$$\frac{c}{3}-\frac{\eta cC_{R}}{3}-3\eta \kappa_4 C_{R}>0.$$ 
Then we fix $\varepsilon>0$ such that 
$$\frac{\kappa_3}{M} -\varepsilon \left(\frac{3\kappa_4C_{R}}{\eta}+\frac{2cC_{R}}{3 \eta}+c\right)>0$$
and \eqref{M/mE3} holds. 
A special case of the Cauchy-Schwarz inequality  $$\left(\sum_{n=1}^6a_n\right)^2\leq \left(\sum_{n=1}^6 1 \right)\left(\sum_{n=1}^6 a_n^2\right)=6\sum_{n=1}^6a_n^2,\,\, \,\,\,\forall\,a_1,...,a_6\in \mathbb{R}$$ implies for any  $i\in\{1,2,3\}$
$$(\rho-6m_{\infty})^2+\sum_{\substack{j=1\\j\neq i}}^3(u_i-u_j)^2+\sum_{j=1}^3(u_i-v_j)^2\geq 6(u_i-m_{\infty})^2$$
and 
$$(\rho-6m_{\infty})^2+\sum_{\substack{j=1\\j\neq i}}^3(v_i-v_j)^2+\sum_{j=1}^3(v_i-u_j)^2\geq 6(v_i-m_{\infty})^2.$$
We sum these inequalities w.r.t $i\in \{1,2,3\}$
\begin{equation}\label{6rho}6(\rho-6m_{\infty})^2+\sum_{\substack{i,j=1\\i\neq j}}^3(u_i-u_j)^2+\sum_{\substack{i,j=1\\i\neq j}}^3(v_i-v_j)^2+2\sum_{i,j=1}^3(v_i-u_j)^2\geq 6 \sum_{i=1}^3\big((u_i-m_{\infty})^2+(v_i-m_{\infty})^2\big).
\end{equation}
We denote 
$$C\colonequals \min\left\{\frac{\kappa_3}{2M} -\varepsilon \left(\frac{3C_{R}\kappa_4}{2\eta}+\frac{cC_{R}}{3 \eta}+\frac{c}{2}\right), \,\,\frac{\varepsilon}{6}\left(\frac{c}{3}-\frac{\eta cC_{R}}{3}-3\eta C_{R}\kappa_4\right) \right\}.$$ 
\eqref{dtEep} and  \eqref{6rho} show \begin{align*}
\frac{d}{dt}\mathrm{E}[u_1(t),&...,v_3(t)] \\ \leq &  -C\int_{\mathbb{T}^3}\Big[6(\rho-6m_{\infty})^2+\sum_{\substack{i,j=1\\j\neq i}}^3(u_i-u_j)^2+\sum_{\substack{i,j=1\\j\neq i}}^3(v_i-v_j)^2+2\sum_{i,j=1}^3(v_i-u_j)^2\Big]dx\\
\leq &  -6C\int_{\mathbb{T}^3}\sum_{i=1}^3\big((u_i-m_{\infty})^2+(v_i-m_{\infty})^2\big)dx\\
\leq &-6CC_1  \mathrm{E}[u_1(t),...,v_3(t)],
\end{align*}
where the last inequality follows from \eqref{M/mE3}.
Gr\"onwall's inequality yields \begin{equation}\label{decE}
\mathrm{E}[u_1(t),...,v_3(t)]\leq e^{-2\tilde{\lambda} t}\mathrm{E}[u_1(0),...,v_3(0)],\,\,\,\,\,\forall \, t\geq 0
\end{equation}
 with $\tilde{\lambda} \colonequals 3 CC_1.$
This estimate and \eqref{M/mE3} provide
\begin{align*}
\sum_{i=1}^3\int_{\mathbb{T}^3}\big(u_i(t,x) &-m_{\infty}\big)^2dx+\sum_{i=1}^3\int_{\mathbb{T}^3}\big({v_i(t,x)}-{m_{\infty}}\big)^2dx\\ \leq & \frac{C_2}{C_1}e^{-2\tilde{\lambda} t} \left[\sum_{i=1}^3\int_{\mathbb{T}^3}\left({u_{0,i}(x)}-{m_{\infty}}\right)^2dx+\sum_{i=1}^3\int_{\mathbb{T}^3}\left({v_{0,i}(x)}-{m_{\infty}}\right)^2dx\right],\,\,\, \forall\, t\geq 0.
\end{align*} 
\end{proof}

\bigskip

\begin{proof}[\bf{Proof of Theorem \ref{main13}}]
Theorem \ref{exist3} provides the solution $u_i, v_i,$ $i\in \{1,2,3\}$ satisfies
$$0\leq u_i,v_i\leq M $$
for all $t\geq 0,$ $ x\in \mathbb{T}^3,$ and $i\in \{1,2,3\}.$ Since we assume $k$ is an interaction rate of type 1, we have 
$$0\leq k(u_1,...v_3,x)\leq \kappa_1,\,\,\forall\,t\geq 0,\, \forall\,x\in \mathbb{T}^3.$$
We consider again the functional $\mathrm{E}[u_1(t),...,v_3(t)]$ in \eqref{E3}. We have estimated $\sum_{i=1}^3 \mathrm{H}[u_i,v_i]$ in Lemma \ref{a.e.H3} $(ii)$.  We want to obtain a similar estimate for $\int_{\mathbb{T}^3}j\cdot \nabla_x \phi dx.$ Let $T>0.$ \eqref{djphi+}, \eqref{dtj<sum}, and \eqref{dfi3} show
\begin{align*}\int_{\mathbb{T}^3}j(T)\cdot \nabla_x \phi(T) dx\leq & \int_{\mathbb{T}^3}j(0)\cdot \nabla_x \phi(0) dx-\left(\frac{c}{3}-\frac{\eta cC_{R}}{3}-3\eta \kappa_1C_{R}\right)\int_0^T\int_{\mathbb{T}^3}(\rho-6m_{\infty})^2 dxdt\nonumber\\&+\left(\frac{3\kappa_1 C_{R}}{\eta}+\frac{2cC_{R}}{3 \eta}+c\right) \int_0^T\int_{\mathbb{T}^3}\sum_{i,j=1}^3(u_i-v_j)^2dxdt,
\end{align*} 
where $\eta>0$ will be fixed later. Combining this estimate with the one in Lemma \ref{a.e.H3} $(ii),$ we get \begin{align}\label{lasE}
\mathrm{E}&[u_1(T),...,v_3(T)]\leq   \mathrm{E}[u_1(0),...,v_3(0)]-\varepsilon \left(\frac{c}{3}-\frac{\eta cC_{R}}{3}-3\eta \kappa_1C_{R}\right)\int_0^T\int_{\mathbb{T}^3}(\rho-6m_{\infty})^2 dxdt\nonumber\\&-\left[(6M)^{\alpha-1}\mathrm{ess\,inf}_{x\in \mathbb{T}^3}k_1(x)-\varepsilon\left(\frac{3\kappa_1 C_{R}}{\eta}+\frac{2cC_{R}}{3 \eta}+c\right)\right] \int_0^T\int_{\mathbb{T}^3}\sum_{i,j=1}^3(u_i-v_j)^2dxdt \nonumber\\
&-\frac{(6M)^{\alpha-1}\mathrm{ess\,inf}_{x\in \mathbb{T}^3}k_1(x)}{2}\int_0^T\int_{\mathbb{T}^3}\Big[\sum_{\substack{i,j=1\\i\neq j}}^3\big((u_i-u_j)^2 +(v_i-v_j)^2\big) \Big]dxdt.
\end{align}
We first fix $\eta>0$ such that
$$\frac{c}{3}-\frac{\eta cC_{R}}{3}-3\eta \kappa_1C_{R}>0.$$ 
Then we fix $\varepsilon>0$ such that  
$$(6M)^{\alpha-1}\mathrm{ess\,inf}_{x\in \mathbb{T}^3}k_1(x)-\varepsilon\left(\frac{3\kappa_1 C_{R}}{\eta}+\frac{2cC_{R}}{3 \eta}+c\right)>0$$
and \eqref{M/mE3} holds.
 Denoting $$C\colonequals \min\left\{\left[\frac{(6M)^{\alpha-1}\mathrm{ess\,inf}_{x\in \mathbb{T}^3}k_1(x)}{2}-\varepsilon\left(\frac{3\kappa_1 C_{R}}{2\eta}+\frac{cC_{R}}{3 \eta}+\frac{c}{2}\right)\right], \,\frac{\varepsilon}{6}\left(\frac{c}{3}-\frac{\eta cC_{R}}{3}-3\eta C_{R}\kappa_1\right) \right\},$$
\eqref{lasE} can be estimated as   \begin{align*}
\mathrm{E}[u_1(T),...,v_3(T)]\leq & \mathrm{E}[u_1(0),...,v_3(0)]\\
&-C\int_0^T\int_{\mathbb{T}^3}\Big[6(\rho-6m_{\infty})^2 +2\sum_{i,j=1}^3(u_i-v_j)^2+\sum_{\substack{i,j=1\\i\neq j}}^3\big((u_i-u_j)^2 +(v_i-v_j)^2\big) \Big]dxdt.
\end{align*}
Then \eqref{6rho} and \eqref{M/mE3} yield \begin{align*}
\mathrm{E}[u_1(T),...,v_3(T)]\leq \mathrm{E}[u_1(0),...,v_3(0)]
-6CC_1\int_0^T\mathrm{E}[u_1(t),...,v_3(t)]dt.
\end{align*}
Grönwall's inequality implies $$\mathrm{E}[u_1(t),...,v_3(t)]\leq e^{-2\tilde{\lambda}_{\alpha}} \mathrm{E}[u_1(0),...,v_3(0)],\,\,\,\,\,\forall\, t\geq 0,$$
where $\tilde{\lambda}_{\alpha}\colonequals 3CC_1.$ Finally, \eqref{M/mE3} provides the claimed result with $\tilde{\Lambda}_{\alpha}\colonequals \frac{C_2}{C_1}.$
\end{proof}

\begin{Remark}
In Theorems \ref{main3} and \ref{main13} the convergence rates $\tilde{\lambda}$ and multiplicative constants $\tilde{\Lambda}$ depend on the $L^\infty$-norms of $u_{i,0}$ and $v_{i,0},$ $i\in \{1,2,3\}.$    As we explained for the one dimensional case in Remark \ref{Rema}, when \eqref{2Eq} is linear (i.e., $k(u_1,u_2,u_3,v_1,v_2,v_3,x)=k_1(x)\in L^{
\infty}(\mathbb{T}^3)$),  we can only require $u_{i,0},\,v_{i,0}\in L^2(\mathbb{T}^3)$  and  $\mathrm{ess\,inf}_{x\in \mathbb{T}^3 } k_1(x)>0$, and we can obtain convergence rates and multiplicative constants independent of $u_{i,0}$ and $v_{i,0}.$  
\end{Remark}
\section{Conclusion and Outlook}\label{sec:conclusion}
In this paper we established exponential convergence of weak solutions for a large class of discrete velocity kinetic equations to their unique global equilibrium in the one and three dimensions. We  tried to make   the interaction rate $k$ as general as possible by only requiring  some growth assumptions. Our main results are:
\begin{enumerate}
\item[(a)] Exponential convergence in the $L^2$ norm for the interaction rates $k$ of type 3, which lets $k$  have singularity, but requires positive lower bounds on initial data; see Theorems \ref{main} and \ref{main3}.
\item[(b)] The same convergence result for the interaction rates $k$ of type 1 with a lower growth assumption, which prevents $k$ being singular but lets it be degenerate and   initial data just be non-negative; see Theorems \ref{main1} and \ref{main13}.
\end{enumerate}
In the proofs,  we constructed  suitable Lyapunov functionals which are equivalent to the $L^2$ norm and satisfy Grönwall's inequality. Here we used the fact that a family of functionals involving convex functions decay under the time evolution of solutions. These functionals are not enough to obtain exponential convergence with explicit convergence rates. Hence, we considered  a specific case, i.e., Boltzmann's entropy, and modified  it by adding suitable terms so that the resulting functional is decreasing and  satisfies Grönwall's inequality with explicitly computable constants.

 As an extension of the present work one could study the long-time behavior of these equations in the whole domain or in a bounded domain with various boundary conditions. The periodicity assumption in this paper is made for technical convenience
and could most likely be removed.
In general, discrete velocity  kinetic equations  have  large number of unknown functions, because they can be considered as  approximation of continuous velocity kinetic equations \cite{Mis} and useful for numerics.  We considered here the one and three dimensional cases. Our  method can be easily generalized to higher dimensions.

Extending our results to other discrete velocity kinetic  equations, would of
course be interesting.
We expect that the techniques used in this paper can be useful to study the long-time behavior  of such equations, for example, the discrete Boltzmann equation \cite{Review}.


\bigskip

\textbf{Acknowledgement.}
 The author is funded  by the Deutsche Forschungsgemeinschaft (DFG, German Research Foundation) under Germany’s Excellence Strategy EXC 2044-390685587, Mathematics Münster: Dynamics-Geometry-Structure.\\

\textbf{Data Availability.} There is no data associated with the paper.\\

\textbf{Conflict of interest.} The author has no conflict of interest to declare.

{}
\end{document}